\documentclass[12pt]{article}
\usepackage{amssymb,amsmath,bbm}
\input{amssym.def}\input{amssym}
\usepackage{graphicx, tikz}
 \usepackage{hyperref}
\usepackage{changebar}

\newcommand{\C}{{\mathbb C}}
\newcommand{\D}{{\mathbb D}}

\newcommand{\N}{{\mathbb N}}

\newcommand{\R}{{\mathbb R}}
\newcommand{\Z}{{\mathbb Z}}
\renewcommand{\S}{{\mathbb S}}
\renewcommand{\P}{{\mathbb P}}

\newcommand\cB{{\mathcal B}}

\newtheorem{theorem}{Theorem}[section]
\newtheorem{thm}[theorem]{Theorem}
\newtheorem{prop}[theorem]{Proposition}
\newtheorem{corollary}[theorem]{Corollary}
\newtheorem{lemma}[theorem]{Lemma}
\newtheorem{remark}[theorem]{Remark}

\newenvironment{proof}{\noindent {\bf Proof.}}{ \hfill $\Box$\\ }

\newcommand\eps{\varepsilon}

\newcommand{\tr}{\operatorname {Tr}}

\title{On multidimensional infinite dihedral group extensions of Gibbs Markov maps}
\date{\today}

\iftrue
\author{Jaime Gomez
	\thanks{Mathematisch Instituut,
		University of Leiden,
		Einsteinweg 55,
		2333 CC Leiden, The Netherlands;
		{\it email address: jaagomezortiz@gmail.com}}
	\and Dalia Terhesiu
	\thanks{Mathematisch Instituut,
		University of Leiden,
		Einsteinweg 55,
		2333 CC Leiden, The Netherlands;
		{\it email address: daliaterhesiu@gmail.com}}
}
\fi

\begin{document}

\maketitle

\begin{abstract}
 We obtain a local central limit theorem for cocycles associated with
  a class of non abelian and non compact group extensions of Gibbs Markov maps. This class consists of 
 multidimensional infinite dihedral groups.  Unlike in the set up of the random walks on groups, we cannot use the convolution of measures on the group and instead we resort to an approach based on irreducible representations. Depending on the dimension of the group, we obtain either mixing, and thus ergodicity, or dissipativity. Also, we obtain the asymptotics of the first return time of the group extension to the origin. 
\end{abstract}

\section{Introduction and summary of the main results}

We consider a class of non abelian and non compact groups, which we refer to as multidimensional infinite dihedral groups. Most probably, this is the easiest class of non abelian and non compact groups for which one can hope to obtain explicit limit theorems. However, obtaining a local central limit theorem for cocycles associated with this class of groups is highly non trivial, even in the set up of (group extensions of) Gibbs Markov maps.
The approach we develop in this work via irreducible representations can be applied to other classes of groups, but each class requires modifications in the form of the irreducible representations. This is the reason why in this work we focus on multidimensional infinite dihedral groups.

The infinite dihedral group is 
$D_\infty=\langle R,S: S^2=e,SRS=R^{-1}\rangle$, where $R$ stands for rotation and $S$ for reflection. The elements of $D_\infty$
are of the form $R^m$ or $S R^m$, $m\in\Z$.
The infinite dihedral group and its multidimensional version can also be
described in terms of semidirect products. That is, 
$D_\infty = \Z/2\Z\ltimes_\phi \Z$ and  its multidimensional version, which we denote by $G_d$, is $G_d= \Z/2\Z\ltimes_\phi \Z^d$. Here $\phi: \Z/2\Z\to \mbox{Aut}(\Z^d)$ is given by $\phi(1)(v)=v$ and $\phi(-1)(v)=-v$ for each $v\in\Z^d$.
For more details on this type of groups, along with details on the irreducible representations which are essential for our approach, we refer to Section~\ref{subs:REVD}. 
In this paper, we restrict to $G_d= \Z/2\Z\ltimes_\phi \Z^d$ but given that the irreducible representations are similar for $\Z/2\Z\ltimes_\phi \mathbb{R}^d$ (see, for instance,~\cite[Example 11]{Palmer} for the case $d=1$), we believe that the treatment is analogous. In Section~\ref{subs:REVD}, we recall/obtain the required description of the dual space of $G_d$ along with the required inversion formula and explain why changing the
group $\Z/2\Z$ by another finite group (relevant when $d\ge 2$) requires a different treatment. This comes down to a different inversion formula. This is possible but every change
of the
group $\Z/2\Z$ requires new arguments. For details on this comment we refer to the paragraphs around Lemma~\ref{lemma: Plancherel-measure} in Section~\ref{subs:REVD}.

 Local central limit theorems (LCLT) for random walks (RW) on $D_\infty$ or $G_d= \Z/2\Z\ltimes_\phi \Z^d$
 under a second moment assumption (along with the aperiodicity) can be extracted from, for instance, RW on  polynomial growth
 discrete groups treated in~\cite{Alexop} (under a finite support assumption for the measure on the group). 
LCLT for RW on $D_\infty$ has also been obtained in~\cite{Tan},
 while looking at the class of affine Weyl groups.
 The LCLT for $D_\infty$ holds in~\cite{Tan} under certain assumptions: the measure $\nu$ has finite support with $\nu(e)>0$ where $e=(1,0)$ is the identity, and it is assumed that $\nu$ is symmetric (in the sense that $\nu(g)=\nu(g^{-1})$ for $g\in D_\infty$). It seems very likely that no symmetry assumption is necessary.
 Recurrence for $D_\infty$ with no symmetry, but a first moment assumption, has been obtained for instance in~\cite[Corollary 6]{Woe}.
 It seems very plausible that LCLT for RW on $G_d$  can be derived from other works, but we could not identify precise explicit results.
 
% Modern formulations for Markov additive processes, partially covering the
% the RW on the infinite dihedral case $D_\infty$, were obtained in~\cite{dL}. We specifically refer to~\cite{KS} which partially covers 
% the RW on $D_\infty$. More precisely, the RW on $D_\infty$
%  decomposes as a RW on $\Z$ whose increments depend on the current reflection state $\eps\in\{0,1\}$; this is exactly the setting of RW with internal degrees of freedom treated in~\cite{KS}, and their LCLT applies to the $\Z$-coordinate of the walk.

 We are interested in a mixing local limit theorem for the infinite dihedral group  $D_\infty$ 
(along with $G_d$, $d\ge 1$) extension of Gibbs Markov maps.
In this introductory section we present the main results in terms of $D_\infty$ extensions of Gibbs Markov maps and point to the generalizations for $G_d$, as obtained in this work.

Let $(X,T,\alpha,\mu)$ be a Gibbs Markov map (see Section~\ref{subsec:GM} for a quick overview).  
Let $T_\psi:X\times D_\infty\to X\times D_\infty$ be the group extension of $T$ by the cocycle $\psi:X\to D_\infty$ defined via
\begin{align*}
 T_\psi(x, g)=(T x, \psi(x) g).
\end{align*}
This map  is invariant under the measure $\hat\mu=\mu\otimes m_{D_\infty}$, where  $m_{D_\infty}$ denotes the Haar measure on  $D_\infty$. Since $D_\infty$ is non compact, $\hat\mu (X\times D_\infty)=\infty$.

For $n\in\N$ and $x\in X$,  define
\begin{align}\label{eq:defpsin}
    \psi_n(x)=\psi(T^{n-1}(x))\cdot\psi(T^{n-2}(x))\cdots\psi(T(x))\cdot\psi(x),
\end{align}
where $\cdot$ stands for the law of the group.
In this paper we are interested in a generalization of LCLT for  the RW on $D_\infty$   to the behaviour of the sequence $\psi_n$.  For this reason, we will reprove the classical result for RW on $D_\infty$ via irreducible representations.
The reason for this is that \emph{we will transport
part of these calculations to obtain a LCLT for $\psi_n$}.

% Once we obtain the LCLT for $\psi_n$, we will be able to obtain the rate of mixing for the infinite measure preserving system $T_\psi$.

Our approach proceeds in two steps.  First we revisit the LCLT
for a general RW on the infinite dihedral group $D_\infty$ (and its multidimensional version $G_d$). Then we generalize this approach (via twisted transfer operators) to understand the behaviour of the sequence $\psi_n$.

As already mentioned, LCLT for such RW can be extracted from previous works, for instance,~\cite{Alexop}.  However,
these proofs exploit the convolution structure and do not
directly adapt to deterministic group extensions of dynamical systems.
For our purposes we therefore require a different point of view.  We prove
the LCLT for the RW on $D_\infty$ (and $G_d$), formulated in Proposition~\ref{proof:proprw}, by analyzing the  irreducible representations of
$D_\infty$ (and $G_d$) (as recalled in Section~\ref{subs:REVD}). The Fourier
transform, described via irreducible representations,  is a  $2\times 2$ matrix with complex entries, and the LCLT
follows from a detailed study of the eigenvalues of the corresponding Fourier transform described as a matrix. The inversion formula is captured in terms of the trace operator. This representation theoretic formulation has a crucial
advantage: by and large, it can be transported to the transfer operator of a
Gibbs Markov map twisted by a $D_\infty$-valued cocycle.  The spectral
perturbation analysis carried out for the random walk thus serves as a model
for the corresponding analysis of the twisted transfer operators, and this is
what ultimately yields the (mixing) local limit theorem for the cocycle $\psi$, as formulated in Theorem~\ref{prop:LCLTGM}. In short, for $D_\infty$, this result says that for $v, w:X\to\C^2$, $v, w$ 'sufficiently regular',
\begin{align*}
 \lim_{n\to\infty}\sqrt n\int_X 1_{\{\psi_n(x)=(1,0)\}} \langle v, w\circ T^n\rangle\, d\mu=\Phi_1(0)\left\langle\int_X v\, d\mu\, , \int_X w\, d\mu\right\rangle,
\end{align*}
where $\Phi_1$ is a suitable Gaussian density.
The LCLT in Theorem~\ref{prop:LCLTGM} is a bit more general phrased in terms of $\psi_n(x)=(\pm 1,r)$, but here we resume to the neutral element
$e=(1,0)$ as to state mixing for $T_\psi$. More importantly, Theorem~\ref{prop:LCLTGM} deals with the whole class of $G_d$ extensions of Gibbs Markov maps.

For $e=(1,0)$ let $1_{\{e\}}$ the indicator function consisting of the neutral element. A classical consequence of the LCLT is that for $\hat v= v\times 
1_{\{e\}}$, $\hat w= w\times 1_{\{e\}}$, with $v$ and $w$ as above,
\(
 \lim_{n\to\infty}\sqrt n\int_{X\times D_\infty} \langle\hat v, \hat w\circ T_\psi^n\rangle\, d\hat\mu\to \Phi_1(0)
 \left\langle\int_{X\times D_\infty} \hat v\, d\mu\,, \int_{X\times D_\infty} \hat w\, d\mu\right\rangle.
\)
Since $T_\psi$ is conservative and ergodic, this implies strong mixing, also known as Krickeberg mixing~\cite{Krickeberg67}. In this set up,
strong mixing means that given  $A,B$ finite union of partition elements  and $g,h\in D_\infty$,
\[
 \lim_{n\to\infty}\sqrt n\hat\mu\left((A\times\{g\})\cap T_\psi^{-n}(B\times\{h\})\right)\to \Phi_1(0)\hat\mu(A\times\{g\})\hat\mu(B\times\{h\}).
\]
For suitable classes of intermittent maps, strong mixing was initially treated in~\cite{Thaler00}  and settled in~\cite{Gouezel11, MT12}, while mixing with complete expansions for $\Z^d$ extensions of dynamical systems (in particular of Gibbs Markov maps) was obtained in~\cite{Pene18}.
For different notions of mixing in infinite measure (such as local-global and global-global) we refer to~\cite{Lenci10} (see also~\cite{DNL21,DN22} and references therein).

We remark that the LCLT for the RW on $D_\infty$ (or $G_d$), namely Proposition~\ref{prop:rw} holds under no symmetry assumption, no finite support of the measure and as required, second moment.
However, for the Gibbs Markov map we assume  a certain mild symmetry assumption (as stated in subsection~\ref{subsec:mres}), though we believe that with some work this assumption can be removed. 
In the sequel (see Theorem~\ref{prop:LCLTGM})  we treat the class of $G_d$ extensions of Gibbs Markov maps and distinguish the behaviour of the extension depending on the dimension $d$: $d\le2$ (mixing, hence ergodic) and $d\ge 3$ (dissipative), see Corollary~\ref{cor:1}. 

Using the main part of the ingredients used in the proof of Theorem~\ref{prop:LCLTGM}, for all $d\ge 1$, we obtain the asymptotics of the first return time to $X\times\{e\}$, where $e$ is the neutral element of $G_d$. The asymptotics varies with $d$ and it is captured in Theorem~\ref{cor:2}.

We recall that (rapid) mixing for compact non abelian group extensions of
hyperbolic systems was first obtained in~\cite{Dolgopyat02}. 
As in~\cite{Dolgopyat02}, perturbed transfer operators (with irreducible 
representations specific to the non abelian and non compact groups we treat)
will play a crucial role.

From a different perspective, related to understanding the pressure function of the (group) extension, several results have been obtained in
~\cite{Ma13, Ja15, DoSh21}. Finally we mention that the LCLT for the RW, namely Proposition~\ref{prop:rw}
on $G_d$ implies a ratio local limit theorem under no symmetry assumption; the case of $G_d$ is rather special, unlike the class of RW treated in~\cite{DoSh24}. Moreover, the LCLT covered in 
Theorem~\ref{prop:LCLTGM} implies the ratio local limit for $G_d$, under symmetry assumptions (though, we believe these can be removed),
enlarging the class treated in~\cite{GomezTerhesiu}.

The plan of the paper is as follows. In Section~\ref{subs:REVD}
we review the class of the groups $G_d$ along with the required background used in the sequel.  In Section~\ref{sec:dinfiid}
we state and prove the LCLT for RW on $G_d$, namely Proposition~\ref{prop:rw}. In Section~\ref{subsec:GM} we recall the basics of Gibbs Markov maps and in subsection~\ref{subsec:mres}
we state the main results: the mixing LCLT Theorem~\ref{prop:LCLTGM},
mixing and dissipativity in Corollary~\ref{cor:1}
and asymptotic of the first return time to the origin, Theorem~\ref{cor:2}. The remainder is allocated to the proofs, mostly to the proof of Theorem~\ref{prop:LCLTGM}.

\section{Overview of groups of the form $\Z/2\Z\ltimes_\phi \Z^d$}\label{subs:REVD}

For $d\geq 1$, we consider the group $G_d= \Z/2\Z\ltimes_\phi \Z^d$ defined as the semidirect product with respect to the group homomorphism $\phi: \Z/2\Z\to \mbox{Aut}(\Z^d)$, $\phi(-1)(v)=-v$ for $v\in \Z^d$, where the group $\Z/2\Z$ is viewed in  multiplicative notation, i.e., $\Z/2\Z=\{1,-1\}$.
% It is known that the homomorphism $\phi$ is determined by a matrix $A\in \mbox{GL}_d(\Z)$ such that $A^2 = I_d$, with $I_d$ the identity matrix of size $d$, and $\phi(1)(v) = v$, $\phi(-1)(v) = Av$ for $v\in\Z^d $.   
% Throughout this section, the dependence on the dimension of $\Z^d$ will be explicit. For the case $d=1$, the only non-abelian group that we can obtain in this manner is given by the infinite dihedral group $D_\infty$, with $A = -1$.

The multiplication on $G_d$ is given by 
\begin{align*}
    (\varepsilon, m)(\varepsilon', r) = (\varepsilon \varepsilon', m+ (-1)^{\frac{1-\varepsilon}{2}}r), \mbox{ for }m,r\in \Z^d, \varepsilon, \varepsilon'\in \Z/2\Z.
\end{align*}
%Note that this is the multidimensional counterpart of the group $D_\infty$ previously studied in this document.

Throughout, for a group $G$ we denote by  $\widehat{G}$ the space of unitary irreducible representations of $G$. Also,  when writing $\theta\in [0,2\pi)^d$ we mean that $\theta$ is a vector considered $\bmod{2\pi}$ in each coordinate, that is, $\theta$ is in the
$d$-dimensional torus.

In this section, we present a precise description of the space $\widehat{G_d}$ by using the Mackey machine, the details can be found in \cite{BaHa19} and \cite[Section 6]{Fol}.
Let $N_d =\{1\}\times \Z^d$.
The irreducible representations of $N_d\cong\Z^d$ are given by $\nu_{{\theta}}(1,r)=e^{i\langle{\theta},{r}\rangle}$, with $\theta$ a vector in $[0,2\pi)^d$ and ${r}\in\Z^d$.
There exists a natural action of $G_d$ on $\widehat{N_d}$ given by
\begin{align}\label{eq:act}
    (g\cdot \nu_\theta)(1,r) = \nu_\theta(g^{-1}(1,r) g), \mbox{ with } g\in G_d, r\in N_d, \nu_\theta\in \widehat{N_d}.
\end{align}
For every $\nu_\theta\in \widehat{N_d}$, the set $(G_d)_{\nu_\theta}$ denotes the stabilizer of $\nu_\theta$ by this action.
Given $\nu_\theta\in \widehat{N_d}$,  we let $H_{\nu_\theta}^d = (G_d)_{\nu_\theta} / N_d$. This follows the notation in  \cite[Section 6.6]{Fol}, except that we keep the $d$  dependence explicitly.
% By \cite[Theorem 6.43]{Fol}, every irreducible representation of $G_d$ is of the form $\mbox{ind}_{(G_d)\nu_\theta}^G(\nu_\theta\rho)$, where $\nu_\theta\in\widehat{N_d}$ and $\rho$ is an irreducible representation of $H_{\nu_\theta}^d$.\\
To describe the irreducible representations we first describe the orbit of $\nu_\theta \in \widehat{N_d}$ and the groups $H_{\nu_\theta}^d$.
For ${\theta}\in [0,2\pi)^d$ and ${r},{m}\in\Z^d$,
\begin{align*}
((-1,{r})\cdot\nu_{{\theta}})(1,{m})= \nu_{{\theta}}((-1,{r})(1,{m})(-1,{r}))=\nu_{\theta}(1,-{m})=e^{i\langle {\theta},-{m}\rangle}=\nu_{-\theta}(1,{m}).
\end{align*}
% where $A^T$ denotes the transpose of $A$.
% Consider $B=A^T-I_d$, and let $K = \mbox{Ker}(B)\cap [0,2\pi)^d$ be its kernel in the $d$-dimensional torus.
% We are not assuming at this point that this space is trivial or not. 
Write
\begin{align}\label{eq:td}
    T_d = \{ 0,\pi\}^d = \{\theta\in [0,2\pi)^d : \theta_i \in \{0,\pi\}, 1\leq i\leq d\}.
\end{align}
The previous expression implies that for each $[0,2\pi)^d\setminus T_d$, $\nu_{\theta}$ has two elements in its orbit and for the case $\theta\in T_d$, $\nu_{\theta}$ is invariant under the action mentioned in~\eqref{eq:act}.
Consequently we have the description of $H_{\nu_\theta}^d$:
\begin{align*}
    H_{\nu_{{\theta}}}^d\cong\begin{cases}
        \Z/2\Z,&\mbox{ if }{\theta}\in T_d,\\
        \{1\},&\mbox{ otherwise. }
    \end{cases}
\end{align*}
We are ready to describe the irreducible representations of $G_d$:\\

$\bullet$ If $\theta\in T_d$, we have $1$-dimensional representations. The group $\Z/2\Z$ has two irreducible representations given by $\rho_1(\gamma)=1$ and $\rho_{-1}(\gamma)=\gamma$, $\gamma\in \Z/2\Z$.
Thus, all the $1$-dimensional irreducible representations of $G_d$ are of the form $\rho_{\theta,\gamma}$, with  $\theta\in T_d\cup K$, $\gamma\in \Z/2\Z$,  defined by $\rho_{\theta,\gamma}(\varepsilon,m) = \nu_\theta(1,m)\rho_\gamma(\varepsilon)$, $m\in \Z^d, \varepsilon\in \Z/2\Z$.\\

$\bullet$ If $\theta\in[0,2\pi)^d\setminus T_d$, we have $2$-dimensional representations. 
The $2$-dimensional irreducible representations of $G_d$ are of the form $\rho_\theta$, where $\theta\in [0,2\pi)^d\setminus T_d$, defined by
\begin{align}\label{eq: 2-dim-rep}
    \rho_{{\theta}}(1,{m})=\begin{bmatrix}
        e^{i\langle {m},{\theta}\rangle}&0\\
        0&e^{-i\langle {m},{\theta}\rangle}
    \end{bmatrix}&&
    \rho_{{\theta}}(-1,{m})=\begin{bmatrix}
        0&e^{i\langle {m},{\theta}\rangle}\\
        e^{-i\langle {m},{\theta}\rangle}&0
    \end{bmatrix}.
\end{align}
In the $2$-dimensional case, $\rho_{{\theta}}$ and $\rho_{-{\theta}}$ are equivalent by \cite[Theorem 6.43]{Fol}. \\
The description used in \eqref{eq: 2-dim-rep} also provides representations for the elements in $T_d$.
These representations are not irreducible, and after changing the base, they can be written as direct sums of $1$-dimensional representations.
Indeed, note that
\begin{align}\label{eq: change-basis-rep}
U\rho_\theta(\varepsilon,m) U^{-1} = \frac{1}{2}\begin{pmatrix}
  e^{i\langle \theta, m\rangle}+e^{-i\langle \theta, m\rangle}& \varepsilon(e^{i\langle \theta, m\rangle}-e^{-i\langle \theta, m\rangle})\\
  e^{i\langle \theta, m\rangle}-e^{-i\langle \theta, m\rangle}&\varepsilon(e^{i\langle \theta, m\rangle}+e^{-i\langle \theta, m\rangle})
\end{pmatrix}, \mbox{ where }   U = \begin{pmatrix}
        \frac{1}{\sqrt{2}} & \frac{1}{\sqrt{2}}\\
        \frac{1}{\sqrt{2}} & -\frac{1}{\sqrt{2}}
    \end{pmatrix}.
\end{align}
In particular, if $\theta\in T_d$, so $e^{i \langle \theta, m \rangle} = \pm 1$, then
\begin{align*}
U\rho_\theta(\varepsilon,m) U^{-1} = \begin{pmatrix}
  e^{i\langle \theta, m\rangle}& 0\\
  0& \varepsilon e^{i\langle \theta, m\rangle}
\end{pmatrix} = \rho_{\theta,1}\oplus \rho_{\theta, -1}.
\end{align*}
Thus, for any $\theta\in T_d$, we can write the $2$ dimensional representation as
\begin{align}\label{eq:opi}
 \rho_\theta=\rho_{\theta,1}\oplus\rho_{\theta,-1},
\end{align}which is a vector in $\C^2$, with $\C^2$ a direct sum of Hilbert spaces (via the matrix $U$).
 This form as
in~\eqref{eq:opi} makes sense when applied to a vector
\(
    \Vec{x}=a \begin{pmatrix}
    \frac{1}{\sqrt{2}}\\\frac{1}{\sqrt{2}}
\end{pmatrix}+  b \begin{pmatrix}
    \frac{1}{\sqrt{2}}\\\frac{-1}{\sqrt{2}}
\end{pmatrix}
\), $a,b\in\C$. That is,
\begin{align*}
    (\rho_{\theta,1}\oplus\rho_{\theta,-1})(\eps,m)(\Vec{x})=a \begin{pmatrix}
    \frac{1}{\sqrt{2}}\\\frac{1}{\sqrt{2}}
\end{pmatrix}+ \varepsilon b \begin{pmatrix}
    \frac{1}{\sqrt{2}}\\\frac{-1}{\sqrt{2}}
\end{pmatrix}.
\end{align*}
The description of the representations given in \eqref{eq:opi} will be used in Sections \ref{s:inv}, \ref{section:spectral properties}, and~\ref{sec:red0}.

\medskip

The space $\widehat{G_d}$ is endowed with the Fell topology, which we briefly recall here using \cite[Proposition 1.C.6]{BaHa19}. 
Let $\pi_d\in \widehat{G_d}$ and $\xi$ a vector in the Hilbert space
$\mathcal{H}_{\pi_d}$ associated to $\pi_d$.
A function of positive type associated to $\pi_d$ is a map defined by
\begin{align*}
\varphi_{\pi_d,\xi}(g)=\langle \pi_d(g)\xi|\xi\rangle_{\mathcal{H}_{\pi_d}},\;\;g\in G_d.
\end{align*}

Let $\varphi$ be a normalized function of positive type associated to $\pi_d$.
A net $(\rho_i)_{i\in I}$ in $\widehat{G_d}$ converges, in the Fell topology, to an element $\pi_d\in \widehat{G_d}$ if there exists a net $(\psi_i)_{i\in I}$ of functions of positive type associated to $\rho_i$ such that $\lim_i\psi_i=\varphi$ pointwise.\\

In our case, for each $\theta\in T_d$ and $\gamma\in \Z/2\Z$, the functions of positive type are given by
\begin{align}\label{eq:positivefunctions-1dim}
    \varphi_{\theta,\gamma, \lambda}(\varepsilon, r) = \rho_{\theta,\gamma}(\varepsilon, r) = e^{i\langle r,\theta\rangle}\rho_\gamma(\varepsilon,0)|\lambda|^2,
\end{align}

and for each ${\theta}\in [0,2\pi)^d\setminus T_d$, 

\begin{align}\label{eq:positivefunctions}
    \varphi_{{\theta},s,t}(\varepsilon,{r})=\begin{cases}
e^{i\langle {r},{\theta}\rangle}|s|^2+e^{-i\langle {r},{\theta}\rangle}|t|^2,&\mbox{ if }\varepsilon=1\\
e^{i\langle {r},{\theta}\rangle}t\overline{s}+e^{-i\langle {r},{\theta}\rangle}s\overline{t},&\mbox{ if }\varepsilon=-1,
    \end{cases}
\end{align}
where $s,t,\in \mathbb{C}$ and $(\varepsilon,r)\in G_d$.

The following lemma will be used to obtain a description for the Plancherel measure on $\widehat{G_d}$  in Lemma \ref{lemma: Plancherel-measure} below.
This description provides the inversion formula \eqref{eq:invnu} and \eqref{eq:invD} below, which are the starting point for obtaining LCLT (for RWs and Gibbs Markov, respectively).
\begin{lemma}
    \label{Prop: induced-planch-meas}
The map 
\begin{align*}
    \Psi:[0,2\pi)^d\setminus T_d\to \widehat{G_d},\mbox{ given by }     \Psi({\theta})=\rho_{{\theta}},
\end{align*}
is continuous, where $[0,2\pi)^d\setminus T_d$ is endowed with the subspace topology inherited from $\mathbb{R}^d$ with the usual topology, and $\widehat{G_d}$ is endowed with the Fell topology.
\end{lemma}
\begin{proof}
Consider a sequence $({\alpha}_i)_{i\in\mathbb{N}}\subseteq  [0,2\pi)^d\setminus T_d$ which converges to ${\alpha}\in [0,2\pi)^d\setminus T_d$. 
We need to prove that $\Psi({\alpha}_i)$ converges to $\Psi({\alpha})$ as $i$ goes to infinity in the Fell topology.
This follows directly from the descriptions of the functions of positive type given in \eqref{eq:positivefunctions-1dim} and \eqref{eq:positivefunctions}.
 Since for each $s,t\in \mathbb{C}\setminus \{0\}$ %with $\mathbb{C}^*$ the dual space of $\C$
 (we do not need all the possible pairs, but just one pair) and  $(\varepsilon,r)\in G_d$, we have that $\varphi_{{\alpha}_i,s,t}(\varepsilon,r)$ converges to $\varphi_{{\alpha},s,t}(\varepsilon,r)$ as $i$ goes to infinity. 
 
 % For the other case, assume that $(\alpha_i)_{i\in \N}\subseteq K\times \{1,-1\}$, which also implies that $\alpha\in K\times \{1,-1\}$.
 % Therefore, we have that $\alpha_i = (\theta_i,\gamma_i)$ and $\alpha = (\theta,\gamma)$, where $\theta_i,\theta\in K$ and $\gamma_i,\gamma\in \{1,-1\}$, $i\in \mathbb{N}$.
 % Using that $\alpha_i\to \alpha$ as $i\to\infty$, we obtain  $\gamma = \gamma_i$ for  every $i\in \mathbb{N}$ big enough.
 % Now, using \eqref{eq:positivefunctions-1dim}, we have that for fixed $(\varepsilon,r)\in G_d$, $\rho_{\gamma_i}(\varepsilon,0)$ is constant for every $i\in \N$ big enough.
 % Hence, $\varphi_{\theta_i,\gamma_i,\xi}$ converges to $\varphi_{\theta, \gamma, \xi}$ pointwise for any $\xi\in \C\setminus\{0\}$ by using the description of these maps given in \eqref{eq:positivefunctions-1dim}.
 
The convergence of these positive functions is condition (iii) in \cite[Proposition 1.C.6]{BaHa19} for the continuity of $\Psi$.
Thus, the conclusion follows.
\end{proof}

Using that $G_d$ is a type I group (virtually abelian and unimodular), the first part of~\cite[Theorem 7.36]{Fol} (Plancherel Theorem) tells us that there exists a unique measure $\omega$ on 
$G_d$ so that for any $V\in L^1(G_d)\cap L^2(G_d)$ (w.r.t.\  Haar measure),
\begin{align}\label{eq:parsD}
 \int_{G_d} |V(g)|^2\, dg
 =\int_{\widehat{G_d}}
 \operatorname{Tr}\left(\hat V(\rho)\;
 {\hat V(\rho)^*}\right) \,d\omega(\rho),
\end{align}
where $\hat{V}(\rho)=\int_G V(g)\rho(g^{-1})\, dg = \sum_{g\in G_d} V(g) \rho(g^{-1})$ is the Fourier transform of $V$ and $\operatorname{Tr}$ stands for the Trace operator. 
The trace operator
 $\operatorname{Tr}\left(\hat V(\rho)\; {\hat V(\rho)}^*\right)$ is well-defined because $\hat V(\rho)$
 is Hilbert-Schmidt and it is square integrable on $G_d$
 (this is the case for any type I, unimodular group, see~\cite[Section 7.4]{Fol}). For a proper definition of the Trace operator we refer to, for instance,~\cite[Appendix A]{Fol}.
 In the set up of $G_d$, one can simply think of $\operatorname{Tr}$ as the trace of a complex valued matrix.

Lemma \ref{Prop: induced-planch-meas} provides a strategy to compute the Plancherel measure of $\widehat{G_d}$ as a slight modification of the pushforward measure of a measure over $[0,2\pi)^d\setminus T_d$.
Given a precise description of the dual space $\widehat{G_d}$ and  of its Plancherel measure, it becomes possible to obtain a concrete inversion formula \eqref{eq:parsD}.
A concrete description of this formula is a key ingredient in the development of the results of this document.

In what follows, we focus on the computation of this formula for a specific subclass of groups considered in this section, namely the multidimensional infinite dihedral groups.
The calculation of a concrete formula \eqref{eq:parsD} is sensitive to variations among  the different groups treated here as it depends on several factors, such as the vector space dimension of $K$ and on the parametrization given in Lemma \ref{Prop: induced-planch-meas}, which are crucial for  the LCLT obtained in this paper. 
For this reason, we focus on semidirect products of $\Z^d$ and $\Z/2\Z$, which constitute the simplest cases to treat. If, for instance, one replaces $\Z/2\Z$ with another finite group (in order to form a semidirect product with $\Z^d$), a completely new computation of the dual space and an analogue of Lemma \ref{Prop: induced-planch-meas} would be required. Moreover, such modifications may give rise to more complicated irreducible representations, possibly of dimension greater than $2$.

% \paragraph{The case  $A=-I_d$: infinite multidimensional dihedral groups.}

% From now on, we assume that the group $G_d = \Z/2\Z\ltimes_\phi \Z^d$ is such that $\varphi$ is determined by the matrix $A=-I_d$. 
% In this case, we have that the space $K$ is an one-element vector space and we may restrict our attention to the two-dimensional irreducible representations of $G_d$ as $K\subseteq T_d$, and $T_d$ is a null set for the Plancherel measure as we will see.
 
By Lemma~\ref{Prop: induced-planch-meas},  the pushforward measure $\Psi_*\sigma$ over $\widehat{G_d}$, where $\sigma$ denotes the Lebesgue measure over $[0,2\pi)^d\setminus T_d$, allows us to describe the Plancherel measure $\omega$ over $\widehat{G_d}$ as discussed in the previous paragraphs.

\begin{lemma}\label{lemma: Plancherel-measure}
    The measure $\hat{\sigma}:=\frac{1}{(2\pi)^d}\Psi_*\sigma$ is the Plancherel measure over $\widehat{G_d}$.
\end{lemma}
\begin{proof}
 First, note that the function of positive type of the irreducible representations of dimension 1 of $G_d$ are not square-summable for any $s\in \mathbb{C}$.
 Thus, \cite[Proposition 18.8.5]{Di77} implies that  the set of irreducible representations of dimension 1 of $G_d$ is null with respect to the Plancherel measure.\\
On the other hand, for every $\theta\in [0,2\pi)^d\setminus T_d$
\begin{align*}
    \tr(\hat{f}(\rho_{{\theta}})\hat{f}(\rho_{{\theta}})^*)=\sum_{m,r\in\Z^d}\left[f(1,r)\overline{f(1,m)}+f(r,-1)\overline{f(-1,m)}\right]\left[e^{-i\langle r-m,{\theta}\rangle}+e^{i\langle r-m,{\theta}\rangle}\right].
\end{align*}
 Therefore, for each $f\in L^1(G_d)\cap L^2(G_d)$
 \begin{align*}
     \int_{\widehat{G_d}}\tr[\hat{f}(\rho)\hat{f}(\rho)^*]d\hat{\sigma}(\rho)&=\dfrac{1}{(2\pi)^d}\int_{[0,2\pi)^d\setminus T_d} \tr[\hat{f}(\rho_{{\theta}})\hat{f}(\pi_{{\theta}})^*]d\sigma({\theta})\\
     &=\sum_{r\in\Z^d}[|f(1,r)|^2+|f(-1,r)|^2]=\int_{G_d} |f(g)|^2 dg.
 \end{align*}
 Thus, \cite[Theorem 18.8.2]{Di77} implies that $\hat{\sigma}$ is the Plancherel measure for $\widehat{G_d}$. 
\end{proof} 

% Using that $\rho_\theta$ and $\rho_{-\theta}$ are equivalent, and $\theta, -\theta$ are two different elements in $[0,2\pi)^d\setminus T_d$, we can also consider as the Plancherel measure the pushforward measure $\frac{1}{(2\pi)^d}\Psi_*\sigma'$, where $\sigma'$ is the Lebesgue measure on $[0,\pi]^d$.
% This form of writing the Plancherel measure on $\widehat{G_d}$ will be used throughout the paper without further notice.\\

Since for $\theta\in [0,2\pi)^d\setminus T_d$ we have a representation of $\rho_\theta$ using matrices, for $f\in L^1(G_d)$, the Fourier transform $\widehat{f}$ is given by
$$\widehat{f}(\rho_\theta)=
\begin{pmatrix} 
\sum_{m\in\Z^d}f(1,m) e^{i\langle m,\theta\rangle } & \sum_{m\in\Z^d}f(-1,m)e^{i\langle m,\theta\rangle} \\
 \sum_{m\in\Z^d}f(-1,m)e^{-i\langle m,\theta\rangle} & \sum_{m\in\Z^d}f(1,m) e^{-i\langle m,\theta\rangle}
\end{pmatrix}.$$
Moreover, as $\widehat{f}(\rho_\theta)$ is 
a matrix, the trace operator $\operatorname{Tr}$ is simply the trace of a matrix.
If $\nu$ is a measure on the group $G_d$, then
\begin{eqnarray}\label{eq:hatnu}
\hat\nu(\rho_\theta) 
 &=& \begin{pmatrix}
    \sum_{m\in\Z^d} \nu(1,m) e^{i \langle m,\theta\rangle} & \sum_{m\in\Z^d} \nu(-1,m) e^{-i \langle m, \theta\rangle}   \\
      \sum_{m\in\Z^d} \nu(-1,m) e^{i \langle m,\theta\rangle} & \sum_{m\in\Z^d} \nu(1,m) e^{-i \langle m,\theta\rangle}
    \end{pmatrix}.
 \end{eqnarray}

Using the explicit form for our Plancherel measure, \cite[Theorem 7.36]{Fol} provides an expression for the inversion formula. That is, for any $f\in L^1(G_d)\cap L^2(G_d)$ and $g\in G_d$,
\(
 f(g)=\frac{1}{(2\pi)^d}\int_{[0,2\pi)^d\setminus T_d}
 \operatorname{Tr}\left(\rho_\theta(g)\;
 \widehat{f}(\rho_\theta)\right)\,d\theta.\)
In particular, given $\nu$ a measure on the group and writing
$(\nu^{*})^n(g)$ for the $n$-fold convolution,
\begin{align}\label{eq:invnu}
 (\nu^{*})^n(g)=\frac{1}{(2\pi)^d}\int_{[0,2\pi)^d\setminus T_d}
 \operatorname{Tr}\left(\rho_\theta(g)\;
 \widehat{\nu}^n(\rho_\theta)\right)\,d\theta
\end{align}
 and for any bounded set $E$ in $G_d$,
 \begin{align}\label{eq:invD}
 1_E(y)=\frac{1}{(2\pi)^d}\int_{[0,2\pi)^d\setminus T_d}
\operatorname{Tr}\left(\rho_\theta(y)\cdot 
 \widehat{1_E}(\rho_\theta)\right)\,d\theta,
\end{align}
where $\widehat{1_E}$ is the Fourier transform of $1_E$.

\paragraph{The particular case  $d=1$: the infinite dihedral group $D_\infty$.}

Here we follow~\cite[Section 3 A]{BaHa19}.
The infinite dihedral group $D_\infty$ can also be represented in matrix form as
\begin{equation}\label{eq:Dnf}
D_\infty = \left\{ \begin{pmatrix} \varepsilon & m \\ 0 & 1 \end{pmatrix}: m \in \Z ,\  \varepsilon=\pm 1\right\},
\text{ and we write }
(\varepsilon, m)=\begin{pmatrix} \varepsilon & m \\ 0 & 1 \end{pmatrix}.
\end{equation}
We note that $R^m$ corresponds to $(1,m)$ and $SR^m$ corresponds to $(-1,m)$, where $R$ stands for rotation and $S$ for reflection.
For $\theta\in [0,2\pi)\setminus T_1$, the unitary representation $\rho_\theta$ are $2$-dimensional and for $m\in\Z$ is given as in \eqref{eq: 2-dim-rep}:
\begin{equation*}
\rho_\theta(1,m) =
\begin{pmatrix} 
e^{im\theta} & 0 \\ 
0 & e^{-im\theta}
\end{pmatrix}, \quad
\rho_\theta(-1,m) =
\begin{pmatrix} 
0 & e^{ im\theta} \\ 
 e^{-im\theta} & 0
\end{pmatrix}.
\end{equation*}
In light of \eqref{eq: change-basis-rep}, when $\theta\in [0,2\pi)\setminus T_1$, these representations are irreducible and the set of these representations has full measure for the Plancherel measure of $\hat{D}_\infty$. On the other hand, if $\theta\in\{0,\pi\}$, $\rho_\theta$ can be written as direct sums of $1$-dimensional representations. 
The group $D_\infty$ has four $1$-dimensional representations given by:
\begin{equation*}\label{eq:rhos}
\begin{cases}
\rho_{0,1}(\varepsilon, m) = 1, \qquad \qquad & \rho_{\pi,1}(\varepsilon, m) = (-1)^m,\\
\rho_{0,-1}(\varepsilon, m) = \varepsilon, &  \rho_{\pi,-1}(\varepsilon, n) = (-1)^m \varepsilon.
\end{cases}
\end{equation*}
The full unitary dual of $D_\infty$ is
\(
\hat{D}_\infty = \{\rho_{0,1}, \rho_{0,-1}, \rho_{\pi,1}, \rho_{\pi,-1} \} \cup \{ \rho_\theta \mid \theta \in [0,2\pi)\setminus T_1 \},
\)
with Plancherel measure given as in Lemma \ref{lemma: Plancherel-measure}.

\section{LCLT for the random walk (RW) on $G_d= \Z/2\Z\ltimes_\phi \Z^d$ (i.i.d. set up)}
\label{sec:dinfiid}

In this section we obtain a LCLT
for the RW on $G_d= \Z/2\Z\ltimes_\phi \Z^d$  via the irreducible representations recalled in Section~\ref{subs:REVD}. We will use part of the argument below directly in obtaining LCLT for the cocycle $\psi$ and this is our motivation for including a new proof. Also, we require no symmetry condition and no finite support for the measure $\nu$ on the group $G_d$,
but we do require some conditions to ensure that the limit covariance matrix is invertible.

In the RW setup, write $S_n=X_1\cdot\ldots\cdot X_n$ and $(\nu^*)^n$ for the $n$-th fold  convolution. Then
\(
 (\nu^{*})^n(g)=\P(\omega\in\Omega: S_n(\omega)=g).
\)

\begin{prop}\label{prop:rw}
Let $X_1,\ldots, X_n$ be independent random variables on some probability space $(\Omega,\mathcal F,\P)$ taking values in $G_d$, distributed according to the measure $\nu$.
Assume that $X_1\in L^{2}(\P)$, equivalently $\nu$ has $L^{2}$ moments.
Assume that
\begin{align}\label{aprw}
 g.c.d.\{n_0\in\N: (\nu^*)^{n_0} (e)>0 \}=1.
\end{align}
%Set $A_{+}=\sum_{m \in \Z} \nu(1,m)$ and $A_{-}=\sum_{m \in \Z} \nu(-1,m)$.
Then for any $(\pm 1, r)\in G_d$, as $n\to\infty$,
\[
  n^{d/2}\P(\omega\in\Omega: S_n(\omega)=(\pm 1, r))=n^{d/2}(\nu^{*})^n(\pm 1,r)\to\Phi\left(0\right),
\]
where $\Phi$ is the density of a Gaussian random variable with mean $0$ and covariance matrix $\Sigma$. To ensure that $\Sigma$ is an invertible matrix
we require some further conditions.
\begin{itemize}
 \item [(a)] The matrix $S=\sum_{m\in\Z^d} (\nu(1,m)+\nu(-1,m) )m\cdot m^T$ is invertible.
 
 \item[(b)] Let  $w_{\pm}=\sum_{m\in \Z^d}\nu( \pm 1, m) m$ be a vector in $\R^d$ and set $A_{-}=\sum_{m\in \Z^d}\nu( - 1, m)$.
 Define the scalars $\beta_0=w_{+}^T S^{-1}w_{+}$, $\beta_1=w_{+}^T S^{-1}w_{-}$, $\beta_2=w_{-}^T S^{-1}w_{-}$.
 We require that $A_{-}>0$ and that $\left(1+\frac{\beta_0}{A_{-}}\right)
\left(1-\frac{\beta_2}{A_{-}}\right)+\left(\frac{\beta_1}{A_{-}}\right)^2\ne 0$.
\end{itemize}
 \end{prop}
The precise form of $\Sigma$ is recorded in the proof: see Section~\ref{prop:rw}.

\subsection{Main ingredients of the proof}
\label{subsec: ingr}
Recalling the formalism in Section~\ref{subs:REVD},
we need to understand the Fourier transform of i.i.d. random variables $X_i$, equivalently the Fourier transform of $\nu$.
Recalling~\eqref{eq:hatnu}, the corresponding Fourier transform is
\begin{eqnarray}\label{eq:mtheta}
M(\theta)=\hat\nu(\rho_\theta)
&=& \begin{pmatrix}
    \sum_{m\in\Z^d} \nu(1,m) e^{i \langle m,\theta\rangle} & \sum_{m\in\Z^d} \nu(-1,m) e^{-i \langle m, \theta\rangle}   \\
      \sum_{m\in\Z^d} \nu(-1,m) e^{i \langle m, \theta\rangle} & \sum_{m\in\Z^d} \nu(1,m) e^{-i \langle m, \theta\rangle}
    \end{pmatrix}.
\end{eqnarray}

Let $\theta\in [0,2\pi)^d\setminus T_d$ with $T_d$ as defined in~\eqref{eq:td} and write $0_d$ for the $d$ dimensional vector
of zeros.
Assuming $\nu(\pm 1, m) > 0$ for at least some $m\in\Z^d$ such that
$e^{i \langle m,\theta\rangle} \neq 1$ (which holds under~\eqref{aprw}), the row  (and also column) sums of $M(\theta)$ are strictly less than $1$ in absolute value. Assume that this maximal absolute value is $c \in (0,1)$, then
$\| M(\theta)v \|_\infty \leq c \| v \|_\infty$. Hence the norm 
$\| M(\theta) \| < c < 1$.
This includes the case $\theta\in T_d\setminus \{0_d\}$.
So, under~\eqref{aprw} there exists $\eps_0\in (0,1)$ and $\delta>0$
so that
\begin{align}\label{pirw}
 \|M(\theta)^n\|=\|\hat\nu(\rho_\theta)^n\|<\eps_0^n,\text{ for all } \theta\in \left([0,2\pi)^d\setminus T_d\right)\setminus B_\delta(0_d).
\end{align}
Here $B_\delta(0_d)=\{\theta\in [0,2\pi)^d:|\theta|\le\delta\}=[-\delta,\delta]^d$, where $|\theta|$
stands for the sup-norm of the vector $\theta$.

We need to compute the trace of $\rho_\theta(\pm 1 , r) M(\theta)^n$
with $r\in\Z^d$ and $\rho_\theta(\pm 1 , r)$ described in~\eqref{eq: 2-dim-rep}
and $M(\theta)$ defined in~\eqref{eq:mtheta},
as to obtain (recall~\eqref{eq:invnu})
\[
 \nu^{*n}(\pm 1 , r)=\P(\omega\in\Omega: S_n(\omega)=(\pm 1 , r))=\frac{1}{(2\pi)^d} \int_{[0,2\pi)^d\setminus T_d}\tr(\rho_\theta(\pm 1 , r) M(\theta)^n)\, d\theta.
\]

We start with computing $M(\theta)^n$. We do so via the general lemma below.
\begin{lemma}\label{lemma:m}
Let $M$ be a $2 \times 2$ matrix with eigenvalues $a_+$ and $ a_-$. Then
\begin{equation}\label{eq:Mn}
M^n =
\begin{cases}
 \frac{a_+^n - a_-^n}{a_+ - a_-}
M + \frac{a_+a_-^n - a_- a_+^n}{a_+ - a_-} I & \text{ if } a_+ \neq a_-,\\[2mm]
n a_+^{n-1} M + (1-n)a_+^n I & \text{ if } a_+ = a_- \neq 0.
\end{cases}
\end{equation}
\end{lemma}

\begin{proof}
We use the Cayley-Hamilton theorem, saying that $M$ solves its own characteristic equation
$$
M^2 - \tau M + \Delta I = 0,
$$
for the trace $\tau = \tr(M)$ and determinant $\Delta = \det(M)$.
Thus the powers of $M$ satisfy the recursive formula
$M^n = \tau M^{n-1} - \Delta M^{n-2}$ for all $n \geq 2$ (which can be verified by an induction argument).
This has the formal solution
\begin{equation}\label{eq:recursol}
 M^n =
\begin{cases} a_+^n A + a_-^n B, \qquad
A+B = I,\quad  a_+ A + a_- B =  M,
& \text{ if } a_+ \neq a_-,\\[1mm]
n a_+^n A + a_+^n B, \qquad
B = I,\quad  a_+ A + a_+ B =  M,
& \text{ if } a_+ = a_- \neq 0.
\end{cases}
\end{equation}
Solving for the matrices $A$ and $B$, and inserting in the first part of \eqref{eq:recursol} gives \eqref{eq:Mn}.
\end{proof}

We will apply this lemma to the matrix $M(\theta)$ defined in~\eqref{eq:mtheta}.
Write $a_+(\theta)$ and $a_{-}(\theta)$ for the corresponding eigenvalues.
Define 
\begin{align}\label{eq:anbn}
 a_n(\theta) = \frac{a_+(\theta)^n - a_-(\theta)^n}{a_+(\theta) - a_-(\theta)}
 \quad \text{ and } \quad
 b_n(\theta) = \frac{a_+(\theta) a_-(\theta)^n - a_-(\theta) a_+(\theta)^n}{a_+(\theta) - a_-(\theta)}.
\end{align}
Then
\begin{eqnarray*}
 \tr( \rho_\theta(1,r) M(\theta)^n) &=&
a_n(\theta) \tr( \rho_\theta( 1,r) M(\theta)) + b_n(\theta)  \tr( \rho_\theta(1,r) I) \\
&=& a_n(\theta) \sum_{m \in \Z^d} \nu(1,m) \left( e^{i\langle \theta, m\rangle} e^{i\langle \theta, r\rangle}+ e^{-i\langle \theta, m\rangle}e^{-i\langle \theta, r\rangle} \right)
+ b_n(\theta) (e^{i\langle \theta, r\rangle} + e^{-i\langle \theta, r\rangle}),
\end{eqnarray*}
and
\begin{eqnarray*}
\tr( \rho_\theta(-1,r) M(\theta)^n) &=&
a_n(\theta) \tr( \rho_\theta(-1,r) M) + b_n  \tr( \rho_\theta(-1,r) I) \\
&=& a_n(\theta) \sum_{m \in \Z^d} \nu(1,m) \left(e^{-i\langle m,\theta\rangle}e^{i\langle r,\theta\rangle}+e^{i\langle m,\theta\rangle}e^{-i\langle r,\theta\rangle}\right).
\end{eqnarray*}
Solving the characteristic polynomial 
\[
 a(\theta)^2-\tau(\theta) a(\theta)+\Delta(\theta) I=0,
\]
where $\tau(\theta)=\tr(M(\theta))$ and $\Delta(\theta)=\det (M(\theta))$.
The quadratic formula gives
\begin{equation}\label{eiv}
 a_{\pm}(\theta)=\frac{\tau(\theta)\pm \sqrt{\tau(\theta)^2-4\Delta(\theta)}}{2}.
\end{equation}
To simplify this, we write expansions in $\theta$ as $\theta\to 0_d$.
At this point, we introduce some further notation. 
Given $m\in\Z^d$ and $\theta\in [0,2\pi)^d\setminus T_d$, let
\begin{align*}
     A_\pm &=\sum_{m\in \Z^d}\nu(\pm 1,m),\\
     B_{\pm}(\theta)&= \sum_{m\in \Z^d}\nu(\pm 1, m) \langle m,\theta\rangle,\\
     C_{\pm}(\theta) &= \sum_{m\in \Z^d}\nu(\pm 1, m) \langle m,\theta\rangle^2.
    \end{align*}
Expanding around $0_d$,
\begin{align*}
 \tau(\theta)=2 A_{+}-C_{+}(\theta)^2(1+o(1))\text{ and } \Delta(\theta)=A_{+}^2-A_{-}^2
 +\Delta_0(\theta)(1+o(1)),
\end{align*}
where
\[
 \Delta_0(\theta)=(B_{+}(\theta)^2-2A_{+}C_{+}(\theta))(1+o(1))-(B_{-}(\theta)^2-2A_{-}C_{-}(\theta))(1+o(1)).
\]
So,
\begin{align*}
 \sqrt{\tau(\theta)^2-4\Delta(\theta)}
 &=2A_{-}-\frac{2}{A_{-}}\left((B_{+}(\theta)^2-B_{-}(\theta)^2)+2C_{-}(\theta)\right)+o(|\theta|^2),
\end{align*}
where $|\theta|$ stands for norm of the vector $\theta$.

Note that $A_{+}+A_{-}=1$.  
Recalling~\eqref{eiv}, we obtain
$$
a_+(\theta) = 1-\frac{1}{A_{-}}\left((B_{+}(\theta)^2-B_{-}(\theta)^2)+2C_{-}(\theta)\right)(1+o(1))= 
e^{-\frac{\langle\Sigma\theta,\Sigma\theta\rangle}{2}(1+o(1))}
$$
for
\begin{align}\label{eq:sigiid}
 \Sigma =\frac{1}{2A_{-}}\sum_{m\in \Z^d}(\nu( 1, m)+\nu(-1, m)) m\cdot m^T+\frac{w_{+}\cdot w_{+}^T-w_{-}\cdot w_{-}^T}{2A_{-}},
\end{align}
where (as defined in the statement of Proposition~\ref{prop:rw})
\begin{align*}
 w_{\pm}=\sum_{m\in \Z^d}\nu(\pm 1, m) m \in \R^d.
\end{align*}
Write $p:=A_{+}-A_{-}<1$.
This together with~\eqref{eiv} gives
$$
a_-(\theta) = p  + O(|\theta|^2).
$$

\paragraph{Leading terms and final form of $\operatorname{Tr}( \rho_\theta(\pm 1,r) M(\theta)^n)$. }
Recalling the expressions of $a_n(\theta)$ and $b_n(\theta)$ in~\eqref{eq:anbn}, and the form of $a_{\pm}(\theta)$, we see that 
\begin{eqnarray*}\label{anbn}
a_n(\theta) &=& \frac{a_+(\theta)^n}{a_+(\theta)-a_-(\theta)}=\frac{e^{-n\frac{\langle\Sigma\theta, \Sigma\theta\rangle}{2}(1+o(1)) }}{1-p+O(|\theta|^2)}+O(p^n),\\
b_n(\theta) &=& \frac{-a_-(\theta)a_+(\theta)^n}{a_+(\theta)-a_-(\theta)}+O(p^n) =-p\frac{e^{-n\frac{\langle\Sigma\theta, \Sigma\theta\rangle}{2}(1+o(1)) }}{1-p+O(|\theta|^2)}+O(e^{-n\frac{\langle\Sigma\theta, \Sigma\theta\rangle}{2} }|\theta|^2)+O(p^n).
\end{eqnarray*}

Since $p<1$, the $O(p^n)$ terms can be ignored.
Thus,
\begin{eqnarray}\label{eq:plus1}
 \tr( \rho_\theta(1,r) M(\theta)^n) &=& \frac{e^{-n\frac{\langle \Sigma\theta, \Sigma\theta\rangle}{2}(1+o(1)) }}{1-p+O(|\theta|^2)} \sum_{m \in \Z^d} \nu(1,m) \left( e^{i\langle \theta, m\rangle} e^{i\langle \theta, r\rangle}+ e^{-i\langle \theta, m\rangle}e^{-i\langle \theta, r\rangle} \right)\\
\nonumber&+&\left(-p\frac{e^{-n\frac{\langle \Sigma\theta, \Sigma\theta\rangle}{2}(1+o(1)) }}{1-p+O|\theta|^2)}+O(e^{-n\frac{\langle \Sigma\theta, \Sigma\theta\rangle}{2} }|\theta|^2)\right) \left(e^{i\langle \theta, r\rangle} + e^{-i\langle \theta, r\rangle}\right)
\end{eqnarray}
and 
\begin{eqnarray}\label{eq:minus1}
\tr( \rho_\theta(-1,r) M(\theta)^n)
&=&\frac{e^{-n\frac{\langle \Sigma\theta, \Sigma\theta\rangle}{2}(1+o(1)) }}{1-p+O(|\theta|^2)}\sum_{m \in \Z^d} \nu(1,m) \left(e^{-i\langle m,\theta\rangle}e^{i\langle r,\theta\rangle}+e^{i\langle m,\theta\rangle}e^{-i\langle r,\theta\rangle}\right)\\
\nonumber&+&O(e^{-n\frac{\langle \Sigma\theta, \Sigma\theta\rangle}{2} }|\theta|^2).
\end{eqnarray}

\paragraph{Invertibility of $\Sigma$.}

By assumption (a) in the statement of Proposition~\ref{prop:rw},
we know that the matrix $S=\sum_{m\in \Z^d}(\nu( 1, m)+\nu(-1, m)) m\cdot m^T$ is invertible.
A direct computation shows that $\det(\Sigma)=\det(S)\left(\left(1+\frac{\beta_0}{A_{-}}\right)
\left(1-\frac{\beta_2}{A_{-}}\right)+\left(\frac{\beta_1}{A_{-}}\right)^2\right)$.

By assumption (b) in the statement of Proposition~\ref{prop:rw}, $\left(1+\frac{\beta_0}{A_{-}}\right)
\left(1-\frac{\beta_2}{A_{-}}\right)+\left(\frac{\beta_1}{A_{-}}\right)^2\ne 0$.
Therefore, $\Sigma$ is invertible.

\subsection{Completing the proof of Proposition~\ref{prop:rw} }
\label{proof:proprw}

As a consequence of~\eqref{pirw}, we can reduce the calculations
to the region $B_\delta(0_d)=[-\delta,\delta]^d$, with $\delta$ small. 
Recall that by~\eqref{pirw}, $\|M(\theta)^n\|<\eps_0^n$ for all 
$\left([0,2\pi)^d\setminus T_d\right)\setminus B_\delta(0_d)$.
So,
\begin{align*}
 (\nu^*)^n(\pm 1,r)=\frac{1}{(2\pi)^d} \int_{[0,2\pi)^d\setminus T_d}  \tr(\rho_\theta(\pm 1,r) M(\theta)^n) \, d\theta
 =\frac{1}{(2\pi)^d} \int_{[-\delta,\delta]^d}  \tr(\rho_\theta(\pm 1,r) M(\theta)^n) \, d\theta+O(\eps_0^n).
\end{align*}
Using~\eqref{eq:plus1},
\begin{align*}
 \int_{[-\delta,\delta]^d} \tr(\rho_\theta(1,r) M(\theta)^n) \, d\theta
&= \int_{[-\delta,\delta]^d}\frac{e^{-n\frac{\langle \Sigma\theta, \Sigma\theta\rangle}{2}(1+o(1)) }}{1-p+O(|\theta|^2)} \sum_{m \in \Z^d} \nu(1,m) \left( e^{i\langle \theta, m\rangle} e^{i\langle \theta, r\rangle}+ e^{-i\langle \theta, m\rangle}e^{-i\langle \theta, r\rangle} \right)\, d\theta
\\
&+ \int_{[-\delta,\delta]^d} \left(-p\frac{e^{-n\frac{\langle \Sigma\theta, \Sigma\theta\rangle}{2}(1+o(1)) }}{1-p+O(|\theta|^2)}+O(e^{-n\frac{\langle \Sigma\theta, \Sigma\theta\rangle}{2} }|\theta|^2)\right) \left(e^{i\langle \theta, r\rangle} + e^{-i\langle \theta, r\rangle}\right)\, d\theta,
\end{align*}
where $\Sigma$ is as defined in~\eqref{eq:sigiid}.
Recall $\sum_{m \in \Z^d}\nu(1,m)=A_{+}$. Expanding further in $\theta$,
\begin{align*}
 \int_{[-\delta,\delta]^d}& \tr(\rho_\theta(1,r) M(\theta)^n) \, d\theta
= (A_{+}-p)\int_{B_\delta(0_d)} \left(e^{i\langle \theta, r\rangle}+ e^{-i\langle \theta, r\rangle} \right)\frac{e^{-n\frac{\langle \Sigma\theta, \Sigma\theta\rangle}{2}(1+o(1)) }}{1-p+O(|\theta|^2)}  d\theta\\
&+\int_{[-\delta,\delta]^d}\frac{e^{-n\frac{\langle \Sigma\theta, \Sigma\theta\rangle}{2}(1+o(1)) }}{1-p+O(|\theta|^2)}O\left(\left|\sum_{m \in \Z^d} \nu(1,m)\langle m,\theta\rangle\right|\right)  d\theta+ \int_{[-\delta,\delta]^d} O(e^{-n\frac{\langle \Sigma\theta, \Sigma\theta\rangle}{2} }|\theta|^2)\, d\theta\\
&=\frac{A_{+}-p}{1-p}\int_{[-\delta,\delta]^d} \left(e^{i\langle \theta, r\rangle}+ e^{-i\langle \theta, r\rangle} \right)e^{-n\frac{\langle \Sigma\theta, \Sigma\theta\rangle}{2}(1+o(1)) } d\theta\\
&+\int_{[-\delta,\delta]^d}e^{-n\frac{\langle \Sigma\theta, \Sigma\theta\rangle}{2}(1+o(1)) }O\left(\left|\sum_{m \in \Z^d} \nu(1,m)\langle m,\theta\rangle\right|\right)  d\theta.
\end{align*}
Since $\nu$ has a first moment, the second integral is 
$O\left(\int_{[-\delta,\delta]^d}|\theta|e^{-n\frac{\langle \Sigma\theta, \Sigma\theta\rangle}{2}(1+o(1)) } d\theta\right)$ and it can be absorbed  in the first.
Thus,
\begin{align*}
  \int_{[-\delta,\delta]^d}\tr(\rho_\theta(1,r) M(\theta)^n) \, d\theta
  =\frac{A_{+}-p}{1-p}\int_{[-\delta,\delta]^d} \left(e^{i\langle \theta, r\rangle}+ e^{-i\langle \theta, r\rangle} \right)e^{-n\frac{\langle \Sigma\theta, \Sigma\theta\rangle}{2}(1+o(1)) } d\theta\, (1+o(1)).
  \end{align*}
  
With a change of variable $\theta\to-\theta$ (in the second term),
\begin{align*}
 \int_{[-\delta,\delta]^d} \left(e^{i\langle \theta, r\rangle}+ e^{-i\langle \theta, r\rangle} \right)e^{-n\frac{\langle \Sigma\theta, \Sigma\theta\rangle}{2}(1+o(1)) } d\theta&=\int_{[-\delta,\delta]^d} e^{i\langle \theta, r\rangle}e^{-n\frac{\langle \Sigma\theta, \Sigma\theta\rangle}{2}(1+o(1)) } d\theta+ \int_{[-\delta,\delta]^d} e^{i\langle \theta, r\rangle}e^{-n\frac{\langle \Sigma\theta, \Sigma\theta\rangle}{2}(1+o(1)) } d\theta\\
 &=2\int_{[-\delta,\delta]^d} e^{i\langle \theta, r\rangle}e^{-n\frac{\langle \Sigma\theta, \Sigma\theta\rangle}{2}(1+o(1)) } d\theta.
\end{align*}
Recall that $A_{+}+A_{-}=1$ and that $A_{+}-A_{-}=p$. So, 
$\frac{A_{+}-p}{1-p}=\frac{1}{2}$.
Thus,
\begin{align*}
 \int_{[-\delta,\delta]^d}\tr(\rho_\theta(1,r) M(\theta)^n) \, d\theta
 =\int_{[-\delta,\delta]^d} e^{i\langle \theta, r\rangle}e^{-n\frac{\langle \Sigma\theta, \Sigma\theta\rangle}{2}(1+o(1)) } d\theta.
\end{align*}

The change of variable $\theta\to \frac{y}{ n^{1/2}}$ gives
\begin{align*}
  n^{d/2}(\nu^*)^n( 1,r)=\frac{1}{(2\pi)^d} \int_{[-n^{1/2}\delta,n^{1/2}\delta]^d} e^{i\frac{\langle y, r\rangle}{n^{1/2}}}e^{-\frac{\langle\Sigma y, \Sigma y\rangle}{2}(1+o(1))} dy\,(1+o(1)) =\Phi\left(-\frac{r}{\sqrt n}\right)\,(1+o(1)),
\end{align*}
where $\Phi$ is the density of the Gaussian with mean $0$ and covariance
matrix $\Sigma$, which is invertible, as explained at the end of subsection~\ref{subsec: ingr}.
                                               
The same approach works for
$(\nu^*)^n( 1,r)= \frac{1}{(2\pi)^d} \int_{B_\delta(0_d)} \tr(\rho_\theta(-1,r) M(\theta)^n) \, d\theta$
starting from equation~\eqref{eq:minus1}.

\section{Overview of Gibbs Markov maps and statement of LCLT}\label{subsec:GM}

Roughly speaking, Gibbs Markov (GM) maps are (infinite branched) uniformly expanding maps with bounded distortion and big images.
We recall the definitions in more detail.
Let $(X,\mu)$ be a topological space and $\mu$ a probability measure on it. Let $T:X\to X$ be a topologically mixing ergodic measure-preserving transformation, piecewise continuous w.r.t.
a countable partition $\alpha=\{a\}$. The map $T$ has full branches in the sense that $T(a)=X$ for all $a\in\{a\}$.

Define $s(x,x')$ to be the least integer $n\ge0$ such that $T^nx$ and $T^nx'$ lie in distinct partition elements.
Since $T$ is expanding,  $s(x,x')=\infty$ if and only if $x=x'$ one obtains that $d_\beta(x,x')=\beta^{s(x,x')}$
for $\beta\in(0,1)$ is a metric.

Let $\varphi=\log\frac{d\mu}{d\mu\circ T}:X\to\R$.
We say that $T$ is a GM map if the following hold w.r.t. the countable partition $\{a\}$:
\begin{itemize}

\parskip = -2pt
\item $T (a)$ is a union of partition elements and $T|_a:a\to T(a)$ is a measurable bijection for each $a\in\{a\}$
such that $T(a)$ is the union of elements of the partition $\bmod\mu$.
\item $\inf_a\mu(T(a))>0$. This is usually referred to as the big image and preimage (BIP) property.
\item
There are constants $C>0$, $\beta\in(0,1)$ such that
$|e^{\varphi(x)}-e^{\varphi(x')}|\le Cd_\beta(x,x') e^{\varphi(x)}$ for all $x,x'\in a$, $a\in\{a\}$.
\end{itemize} The invariant measure $\mu$ is known
to have the Gibbs property: there exists $C>0$ such that $C^{-1}\mu(a)\le e^{\varphi_n(x)}\le C\mu(a)$
for all $x\in a$ and $a\in\alpha_n=\bigvee_{j=0}^{n-1}T^{-j}\alpha$.
See, for instance,~\cite[Chapter 4]{Aaronson} and~\cite{AD01}
for background on Gibbs-Markov maps.

\subsection{Statement of the main results}
\label{subsec:mres}

Let $\psi:X\to G_d$ with $G_d= \Z/2\Z\ltimes_\phi \Z^d$ as described in Section~\ref{subs:REVD}. Because we work with GM maps and $G_d$ is discrete, we assume, without loss of generality, that $\psi$ is constant on each $a\in\alpha$.
Since $\psi$ takes values in $G_d$,
we write
$$
\psi(x)=(\varepsilon(x), \psi^{\Z^d}(x)),
$$
where $\eps(x)=\pm 1$, and $\psi^{\Z^d}$ is the $\Z^d$ component of $\psi$.

In the i.i.d. case treated in Section~\ref{sec:dinfiid}
 no symmetry assumption is required to obtain the LCLT. It seems plausible
 in the dynamical systems set up  one can get rid of the symmetry assumption below, but it does simplify the exposition.

The  symmetry assumption that we require is that 
there exists a $\mu$-preserving involution $S:X\to X$ so that
 
 \begin{equation}\label{eq:assS}
   T=T\circ S, \quad\eps\circ S=-\eps \quad \text{ and }  S \text{ is a bijection on } \{a\}.
  \end{equation}
  The second part of~\eqref{eq:assS} and $T$-invariance of $\mu$
  imply that
  \(
  \mu(\{\eps=1\})=\mu(\{\eps=-1\})=1/2.
 \)
The last assumption on $S$ is saying that by composing with $S$ we move from one full branch to another or from one partition element to another partition element in a unique way. By looking at $Sx$ instead of $x$ we just swap from one  full branch to another full branch or from one partition element to another partition element.

  To obtain a LCLT for $\psi$, we require that
  \begin{equation}\label{eq:2plus}
   \psi^{\Z^d}\in L^{2+\delta^*}(\mu),\quad\text{ for some }\delta^*>0.
  \end{equation}
 It is known that under the assumption $\psi^{\Z^d}\in L^{2}(\mu)$, things are more difficult, see~\cite{Gou04}. One could work with $\psi^{\Z^d}\in L^{2}(\mu)$, but to simplify we assume~\eqref{eq:2plus}.

The LCLT we are after reads as 

\begin{thm}\label{prop:LCLTGM}
 Let $\psi:X\to G_d$. Suppose that the symmetry assumption~\eqref{eq:assS} holds.
 Assume~\eqref{eq:2plus} and further suppose that a suitable aperiodicity condition (as in~\eqref{eq:cohomology-like}) holds. Then as $n\to\infty$,
 \[
  n^{d/2}\mu\left(x\in X: \psi_n(x)=(\pm 1,r)\right)\to\Phi_1\left(0\right),
\]
where $\Phi_1$ is the density of the Gaussian with mean $0$ and covariance\footnote{The precise form of $\Sigma_1$ is provided below equation~\eqref{eigenvformula}.}  matrix $\Sigma_1^2$. 
More generally, for $v, w:X\to\C^2$ `sufficiently regular',
\begin{align*}
 \lim_{n\to\infty}n^{d/2}\int_X 1_{\{\psi_n(x)=(\pm 1,r)\}} \langle v, w\circ T^n\rangle\, d\mu=\Phi_1(0)\left\langle\int_X v\, d\mu\, , \int_X w\, d\mu\right\rangle,
\end{align*}
where $\langle\cdot ,\cdot \rangle$ is the usual scalar product.
\end{thm}

By `sufficiently regular` we mean that is an a suitable Banach space,
more precisely $v, w\in\cB(X,\C^2)$ as defined in Section~\ref{s:inv}.
 
We record some consequences of the LCLT for $\psi$.

\begin{corollary}\label{cor:1}Assume the set up of Theorem~\ref{prop:LCLTGM}.
 If $d=1,2$, then $T_\psi$ is conservative, ergodic and strong (Krickeberg) mixing.
 
 If $d\ge 3$, then $T_\psi$ is dissipative (that is, not conservative, and in fact, not recurrent).
\end{corollary}

Using a main part of the ingredients used in the proof of Theorem~\ref{prop:LCLTGM} along with a renewal equation, namely a discrete version of the renewal equation in~\cite{Tho16}, we obtain the asymptotic of the first return time to the origin.

\begin{thm}\label{cor:2}Assume the set up of Theorem~\ref{prop:LCLTGM}.
Let $Y=X\times\{e\}$, where $e=(1,0_{\Z^d})$ is the neutral element of $G_d$, and let $\tau:Y\to\N$ be the first return time to $Y$.
 Let $\tilde\nu=\mu\otimes\delta_e$.
Then 
\begin{itemize}
 \item If $d=1$, then $\tilde\nu(\tau\ge n)=C_1 n^{-1/2}(1+o(1))$, for some $C_1>0$. 
 
 \item If $d=2$, then $\tilde\nu(\tau\ge n)=C_2 (\log n)^{-1}(1+o(1))$,
  for some $C_2>0$. 
  
  \item If $d\ge 3$, then $\lim_{n\to\infty}\tilde\nu(\tau\ge n)$ is a constant and $\sum_n \tilde\nu(\tau=n)< 1$.
\end{itemize}

\end{thm}

Corollary~\ref{cor:1} is a standard consequence of LCLT (ergodicity, conservativity or dissipativity follows as in~\cite{AD01}, see also references therein). For $d=1,2$, strong mixing is also a standard consequence of LCLT. Once a LCLT, namely Theorem~\ref{prop:LCLTGM}, is obtained, the mechanism for obtaining strong mixing is the same as if working with $\Z^d$ extensions.

The proof of Theorem~\ref{cor:2} is included in Section~\ref{sec:compdih}. 
One could try to follow the arguments in~\cite{T22} to obtain the tail of $\tilde\nu(\tau=n)$, 
but that will require much lengthier arguments, which we omit.

The proof of Theorem~\ref{prop:LCLTGM}
is the difficult part; it constitutes the bulk of the remainder of the paper and it is concluded in Section~\ref{sec:compdih}.

\section{Inversion formula in the Gibbs Markov set up}
\label{s:inv}
For the purpose of this paper we choose to work with functions $v:X\to \C^2$.
The reason for this choice will become transparent in subsection~\ref{sec:inversion} below. 

\subsection{Perturbed transfer operator}

With the notation used in Section~\ref{subs:REVD},
let $\rho_\theta$, $\theta\in [0,2\pi)^d$.
Given $\psi:X\to G_d$, define 
the perturbed transfer operator $L_{\rho_\theta} v= L(\rho_\theta\circ\psi\, v)$, $v:X\to \C^2$.

Recall that
$
\psi(x)=(\varepsilon(x), \psi^{\Z^d}(x)),
$
where $\eps(x)=\pm 1$ and $\psi^{\Z^d}$ is the $\Z^d$ component of $\psi$. Then
$\rho_\theta\circ\psi(x)=\rho_\theta(\varepsilon(x), \psi^{\Z^d}(x))$.

\textbf{For $\theta\in [0,2\pi)^d\setminus T_d$, with $T_d$ as defined in~\eqref{eq:td}},
\[
\rho_\theta(1, \psi^{\Z^d}(x)) =
\begin{pmatrix} 
e^{i\langle \theta, \psi^{\Z^d}(x)\rangle} & 0 \\ 
0 & e^{-i\langle\theta,\psi^{\Z}(x)\rangle}
\end{pmatrix}, \quad
\rho_\theta(-1, \psi^{\Z^d}(x))=
\begin{pmatrix} 
0 & e^{ i\langle\theta,\psi^{\Z^d}(x)\rangle} \\ 
e^{-i\langle\theta,\psi^{\Z^d}(x)\rangle} & 0
\end{pmatrix}. 
\]

The operator $L_{\rho_\theta}$ can be iterated, using test-functions $v,w : X \to \C^2$  as follows (where $\cdot$ between vectors means inner product):
\begin{eqnarray*}
 \int_X L_{\rho_\theta} v \cdot w \, d\mu &=&
 \int_X \left(  \rho_\theta \circ \psi  v \right) \cdot w \circ T \, d\mu
 = \sum_{a \in \alpha}  \int_a \left( \rho_\theta \circ \psi(a) v \right) \cdot w \circ T \, d\mu. \\
 \int_X L_{\rho_\theta}^2 v \cdot w \, d\mu 
 &=& \int_X \left( \rho_\theta \circ \psi L_{\rho_\theta} v \right) \cdot w \circ T \, d\mu
 = \sum_{a \in \alpha} \int_a \left( \rho_\theta \circ \psi(a)  L_{\rho_\theta} v \right) \cdot w \circ T \, d\mu \\
 &=& \sum_{a \in \alpha_2} \int_a \left( \rho_\theta \circ \psi(a)
  \rho_\theta \circ \psi(Ta)  v\right) \cdot w \circ T^2 \, d\mu
 =\int_X \rho_\theta\circ\psi_2 v \cdot w\circ T^2 \, d\mu.
\end{eqnarray*}
And $\int_X L_{\rho_\theta}^n v \cdot w \, d\mu=\int_X \rho_\theta\circ\psi_n v \cdot w\circ T^n \, d\mu$  by induction.
The same applies to $f,g:X\to\C$, $\int_X L_{\rho_\theta}^n (f)  g \, d\mu=\int_X \rho_\theta\circ \psi_n f  (g\circ T^n) \, d\mu$.

The pointwise formula for the perturbed transfer operator,
$$
 L_{\rho_\theta} v(x)= L(\rho_\theta\circ\psi\, v)(x)
 =\sum_{y\in T^{-1}x} e^{\varphi(y)}\rho_\theta\circ\psi(y)\, v(y),
$$
 is
\begin{align*}L_{\rho_\theta} v(x) &=\sum_{y\in T^{-1}x, \psi(y)=(1,\psi^{\Z^d}(y))} e^{\varphi(y)}
\begin{pmatrix} 
e^{i\langle\theta,\psi^{\Z^d}(x)\rangle} & 0 \\ 
0 & e^{-i\langle\theta,\psi^{\Z^d}(y)\rangle}
\end{pmatrix}v(y)\\
&+\sum_{y\in T^{-1}x, \psi(y)=(-1,\psi^{\Z^d}(y))}e^{\varphi(y)}
\begin{pmatrix} 0&
 e^{i\langle\theta,\psi^{\Z^d}(y)\rangle} \\ 
e^{-i\langle\theta,\psi^{\Z^d}(y)\rangle}& 0
\end{pmatrix}v(y).
\end{align*}
where $v(y)=(v_1(y), v_2(y))^T\in \C^2$ and $\psi^{\Z^d}$ is the $\Z^d$ component of $\psi$.

\textbf{For $\theta\in T_d$, we have one dimensional irreducible representations (characters).}

 Let $\theta\in T_d$.
 By~\eqref{eq:opi}, $\rho_\theta=\rho_{\theta,1}\oplus\rho_{\theta,-1}$.
 Recall from the explanation after~\eqref{eq:opi} that this expression makes sense when we apply it to elements in
 $$
 \C^2=\C\begin{pmatrix}
 \frac{1}{\sqrt{2}}\ \\ 
 \frac{1}{\sqrt{2}}
 \end{pmatrix}\bigoplus\C\begin{pmatrix} 
 \frac{1}{\sqrt{2}}\ \\ 
 -\frac{1}{\sqrt{2}}
 \end{pmatrix}.
 $$
Recall $v:X\to\C^2$, $v=(v_1, v_2)$ and that $v=v_1 e_1 +v_2 e_2$, with $(e_1, e_2)$ standard Euclidean basis.
We express $v(x)$ as a linear combination of two orthonormal vectors in $\C^2$, via $u_+=\frac{1}{\sqrt 2}(1, 1)^T=\frac{1}{\sqrt 2}(e_1+e_2)$ (the `diagonal' direction) and  $u_-=\frac{1}{\sqrt 2}(1, -1)^T=\frac{1}{\sqrt 2}(e_1-e_2)$ (the `anti-diagonal' direction).
These form an orthonormal basis  for $\C^2$. A direct calculation shows that the coordinates of $v(x)$ in the  basis $\{u_+,u_-\}$ are $(v^{+}(x),v^{-}(x))=(\frac{v_1(x)+v_2(x)}{\sqrt 2}, \frac{v_1(x)-v_2(x)}{\sqrt 2})$.
 
 As a consequence, for $\theta\in T_d$, we can write $L_{\rho_\theta} =L_{\rho_{\theta,1}}\oplus L_{\rho_{\theta,-1}}$,
 where
 \begin{align}\label{eq:operatorchangebasis1}
 \begin{cases}
  L_{\rho_{\theta,1}}v^{+}(x) = \sum_{y \in T^{-1}(x)} e^{\varphi(y)} e^{i\langle \theta,\psi^{\Z^d}(y)\rangle} v^{+}(y), \\[2mm]
  L_{\rho_{\theta,-1}} v^{-}(x)= \sum_{y \in T^{-1}(x)} e^{\varphi(y)} e^{i\langle \theta,\psi^{\Z^d}(y)\rangle} \eps(y)v^{-}(y) .
 \end{cases}
\end{align}

Let $_{da}$ stand for diagonal-anti-diagonal basis.
As a matrix-representation w.r.t.\ this basis we get, for $n\ge 1$,
$$
(L_{\rho_\theta}^n)_{da}\begin{pmatrix}
                    v^{+}\\
                   v^{-} 
                  \end{pmatrix} = \begin{pmatrix}
                   L_{\rho_{\theta,1}}^n  & 0 \\
                   0 & L_{\rho_{\theta,-1}}^n
                  \end{pmatrix}_{da}
                  \begin{pmatrix}
                    v^{+}\\
                   v^{-} 
                  \end{pmatrix}.
$$
The entire analysis will be carried in Euclidean coordinates.
$(L_{\rho_\pi}^nv)_{da}$ will be in the Euclidean representation by conjugating via the matrix $U = U^{-1} = \frac{1}{\sqrt{2}} \binom{1 \ 1}{1\ -1}$ (as introduced in~\eqref{eq: change-basis-rep}).
This gives
\begin{align}\label{changebmult1111111111}
 (L_{\rho_\theta}^n)_{euc} = U^{-1} (L_{\rho_\theta}^nv)_{da} U
  = \frac12 \begin{pmatrix}
                   L_{\rho_{\theta,1}}^n  + L_{\rho_{\theta,-1}}^n  &  L_{\rho_{\theta,1}}^n- L_{\rho_{\theta,-1}}^n  \\
                    L_{\rho_{\theta,1}}^n - L_{\rho_{\theta,-1}}^n  &  L_{\rho_{\theta,1}}^n + L_{\rho_{\theta,-1}}^n 
                  \end{pmatrix}_{euc}.
\end{align}

The previous paragraphs tells us we need to use different types of analysis, for $\theta\in [0,2\pi)^d\setminus T_d$ working directly with $L_{\rho_\theta}$ and change of basis
for $\theta\in T_d$.

\subsection{Inversion formula}
\label{sec:inversion}

Let $E_{n,x}=\{\psi_n(x)\}$ for $x\in X$. Using the definition of the Fourier transform \eqref{eq:hatnu} for $1_{E_{n,x}}$ we obtain
\begin{align*}
    \widehat{1_{E_{n,x}}}(\rho_\theta)=\begin{pmatrix}
        \sum_{m\in \Z^d}1_{E_{n,x}}(1,m)e^{i\langle m,\theta\rangle}&
        \sum_{m\in \Z^d}1_{E_{n,x}}(-1,m)e^{i\langle m,\theta\rangle}\\
        \sum_{m\in \Z^d}1_{E_{n,x}}(-1,m)e^{-i\langle m,\theta\rangle}&
        \sum_{m\in \Z^d}1_{E_{n,x}}(1,m)e^{-i\langle m,\theta\rangle}
    \end{pmatrix}.
\end{align*}
Note that $1_{E_{n,x}}(1,m)$ is $1$ when $\psi_n(x)=(1,m)$ and $0$ otherwise.
Therefore, we have that 
\begin{align*}
     \widehat{1_{E_{n,x}}}(\rho_\theta)=
     \begin{cases}
         \begin{pmatrix}
        e^{i\langle \psi_n^{\Z^d}(x),\theta\rangle}&
        0\\
        0&
        e^{-i\langle \psi_n^{\Z^d}(x),\theta\rangle}
    \end{pmatrix}&\mbox{ if }\psi_n(x)=(1,\psi_n^{\Z^d}(x))\\[3mm]
    \begin{pmatrix}
        0&e^{i\langle \psi_n^{\Z^d}(x),\theta\rangle}\\
        e^{-i\langle \psi_n^{\Z^d}(x),\theta\rangle}&0
    \end{pmatrix}&\mbox{ if }\psi_n(x)=(-1,\psi_n^{\Z^d}(x))
     \end{cases}. 
\end{align*}
That is,
\begin{align*}
 \widehat{1_{E_{n,x}}}(\rho_\theta)=\rho_\theta\circ \psi_n(x)
 =\begin{pmatrix}
        1_{\varepsilon=1}(x)e^{i\langle \psi_n^{\Z^d}(x),\theta\rangle} & 1_{\varepsilon=-1}(x)e^{i\langle \psi_n^{\Z^d}(x),\theta\rangle}\\
    1_{\varepsilon=-1}(x)e^{-i\langle \psi_n^{\Z^d}(x),\theta\rangle} & 1_{\varepsilon=1}(x)e^{-i\langle \psi_n^{\Z^d}(x),\theta\rangle}
    \end{pmatrix}.
\end{align*}
Note that given $e_1=(1,0)^T$ and $e_2=(0,1)^T$,
\begin{align*}
 L_{\rho_\theta}^n(1e_1)(x) = 
 \begin{pmatrix}
1_{\varepsilon=1}(x) 
e^{i\langle \psi_n^{\Z^d}(x),\theta\rangle}\\
 1_{\varepsilon=-1}(x) e^{-i\langle \psi_n^{\Z^d}(x),\theta\rangle}
 \end{pmatrix},\quad
 \quad 
     L_{\rho_\theta}^n(1e_2)(x) = \begin{pmatrix}
       1_{\varepsilon=-1}(x) e^{i\langle \psi_n^{\Z^d}(x),\theta\rangle}\\
        1_{\varepsilon=1}(x) e^{-i\langle \psi_n^{\Z^d}(x),\theta\rangle}
 \end{pmatrix}.
\end{align*}
Thus,
\[
 \widehat{1_{E_{n,x}}}(\rho_\theta)=[L_{\rho_\theta}^n (1e_1(x))\quad\quad  L_{\rho_\theta}^n (1e_2(x))]: =M_{n,\theta}(1,1)^T(x).
\]
Therefore,
\begin{align}\label{eq:invdihGM2}
 1_{\{x\in X: \psi_n(x)=g\}}=\frac{1}{(2\pi)^d}\int_{\theta\in [0,2\pi)^d\setminus T_d}
 \operatorname{Tr}\left(\rho_{\theta}(g)\cdot 
 M_{n,\theta}(1, 1)^T(x)\right)\,d\theta
\end{align}
Integrating on both sides w.r.t.\ $\mu$ gives

\begin{align}\label{eq:invdihGM}
 \mu(\{x\in X: \psi_n(x)=g\})=\frac{1}{(2\pi)^d}\int_{[0,2\pi)^d\setminus T_d}
 \operatorname{Tr}\left(\rho_\theta(g)\, M_{n,\theta}^* (1,1)^T\right)\, d\theta,
\end{align}
where
\begin{align*}
      M_{n,\theta}^*(1,1)^T=\left[\int_X L_{\rho_\theta}^n (1e_1)\, d\mu\quad\quad  \int_X L_{\rho_\theta}^n (1e_2)\, d\mu\right].
     \end{align*}
Consider $v\in L^1(\mu)$,
 $w\in L^\infty$, $v,w:X\to \C^2$, $v=(v_1, v_2)^T$,$w=(w_1, w_2)^T$. Repeating the argument for the pointwise formula, this time with $1_{E_{n,x}} v_j(x) w_j(T^nx)$, $j=1,2$, integrating over the space  and
 using  that $\int_X L_{\rho_\theta}^n v_j \cdot w_j \, d\mu=\int_X \rho_\theta(\psi_n) v_j \cdot w_j\circ T^n \, d\mu$,
 we obtain
 \begin{align*}
 \int_X 1_{\{\psi_n(x)=g\}} v_1\, (w_1\circ T^n)\, d\mu=\frac{1}{(2\pi)^d}\int_{[0,2\pi)^d\setminus T_d}
 \operatorname{Tr}\left(\rho_\theta(g)M_{n,\theta}^*(v_1, w_1)\right)\, d\theta,
\end{align*}
 where
 \begin{align}\label{eq:mnstar}
      M_{n,\theta}^*(v_1, w_1)=\left[\int_X L_{\rho_\theta}^n (v_1e_1)\cdot w_1\, d\mu\quad\quad  \int_X L_{\rho_\theta}^n (v_1e_2)\cdot w_1\,d\mu\right]
     \end{align}
 and a similar formula for $v_2,w_2$.
 
 Using the formula for $\int_X 1_{\{\psi_n(x)=g\}} v_1\, w_1\circ T^n\, d\mu$ and $\int_X 1_{\{\psi_n(x)=g\}} v_2\, w_2\circ T^n\, d\mu$, we obtain

\begin{align}\label{eq:invdihGMfun}
 \int_X 1_{\{\psi_n(x)=g\}} \langle v, w\circ T^n\rangle\, d\mu=\frac{1}{(2\pi)^d}\int_{[0,2\pi)^d\setminus T_d}
 \operatorname{Tr}\left(\rho_\theta(g)M_{n,\theta}^*(v, w)\right)\, d\theta,
\end{align}
where
\begin{align*}
      M_{n,\theta}^*(v, w)=M_{n,\theta}^*(v_1, w_1)+ M_{n,\theta}^*(v_2, w_2).
     \end{align*}

\section{$L_{\rho_\theta}$ as an operator on $\cB_\beta(X, \C^2)$ }\label{section:spectral properties}

 Given $v:X\to\C^2$, write $v= (v_1, v_2)$. Define the seminorm 
\[
D_av=\max_{j=1,2}\sup_{x,x'\in a,\,x\neq x'}|v_j(x)-v_j(x')|/d_\beta(x,x'),\qquad |v|_\beta=\sup_{a\in\{a\}}D_a v.
\]In the same way, we let the norm of $L^\infty(\mu)$ be defined by 
$\|v\|_{\infty}=\max_{j=1,2}\|v_j\|_{\infty}$.
The space $\cB_\beta(X,\C^2)\subset L^\infty(\mu)$ consisting of the functions $v:X\to\C^2$  with norm
$\|v\|_{\cB_\beta(X,\C^2)}=\|v\|_{\infty}+|v|_\beta<\infty$ is a Banach space.

\subsection{$L_{\rho_\theta}$ acting on $\cB_\beta(X, \C^2)$ for $\theta\in\theta\in [0,2\pi)^d\setminus T_d$.}
\label{outsidetd}

For $\theta \in [0,2\pi)^d\setminus T_d$, the DF (Doeblin Fortet or Lasota Yorke) inequality $\|L_{\rho_\theta}^n v\|_{\cB_\beta(X,\C^2)}\le \gamma_0^n |v|_{\beta}
+ C\|v\|_\infty$, $\gamma_0\in (0,1)$, $C>0$  can be proved using the same pattern as the one in the proof of~\cite[Proof of Proposition 1.4]{AD01},
using the definition of the norm and the seminorm in $\cB_\beta(X,\C^2)$ provided in the previous paragraph.

Using that $\psi$ is constant on partition elements, we obtain the following 
continuity property of $L_{\rho_\theta}: \cB_\beta(X,\C^2)\to \cB_\beta(X,\C^2)$:

\begin{lemma}\label{lemma:Dcont} Let $\theta,\theta'\in [0,2\pi)^d\setminus T_d$.
Then
$\|L_{\rho_\theta} -L_{\rho_\theta'}\|_{\cB_\beta(X,\C^2)}\to 0$ as $\theta\to\theta'$.
\end{lemma}

\begin{proof}
By definition,
\begin{align*}
 &L_{\rho_\theta}v(x)
-L_{\rho_{\theta'}}v(x)\\
&=\sum_{y\in T^{-1}x,\,\psi(y)=(1,\psi^{\Z^d}(y)) } e^{\varphi(y)}
\begin{pmatrix} 
e^{i\langle\theta,\psi^{\Z^d}(y)\rangle}-e^{i\langle\theta',\psi^{\Z^d}(y)\rangle} & 0 \\ 
0 & e^{-i\langle\theta,\psi^{\Z^d}(y)\rangle}-e^{-i\langle\theta,\psi^{\Z^d}(y)\rangle}\\
\end{pmatrix}v(y)\\
&+\sum_{y\in T^{-1}x,\,\psi(y)=(-1,\psi^{\Z^d}(y)) } e^{\varphi(y)}
\begin{pmatrix} 
0 & e^{i\langle\theta,\psi^{\Z^d}(y)\rangle}-e^{i\langle\theta',\psi^{\Z^d}(y)\rangle}  \\ 
 e^{-i\langle\theta,\psi^{\Z^d}(y)\rangle}-e^{-i\langle\theta,\psi^{\Z^d}(y)\rangle}& 0\\
\end{pmatrix}v(y).
\end{align*}
After applying the matrix $\rho_\theta\circ \psi$
(in either case, $\psi(y)=(1,\psi^{\Z^d}(y))$
or $\psi(y)=(-1,\psi^{\Z^d}(y))$) to $v(y)=(v_1(y), v_2(y))^T$ we get a vector.
So, we can operate with taking the maximum on components, according to the definition of the norm on $\cB_\beta(X,\C^2)$.

Choose $\delta > 0$ arbitrary, and a subcollection $\alpha' \subset \alpha$ such that $\sum_{a \in \alpha'} \mu(a) < \delta$.
Then  $|e^{ i\langle\theta,\psi^{\Z^d}(y_a)\rangle}-e^{ i\langle\theta',\psi^{\Z^d}(y_a)\rangle}|\to 0$,
as $\theta\to\theta'$ for
$y_a = T^{-1}(y) \cap a$,
uniformly over all $x \in X$ and $a \in \alpha'$.
The same holds for $|e^{-i\langle\theta,\psi^{\Z^d}(y)\rangle}-e^{- i\langle\theta',\psi^{\Z^d}(y)\rangle}|$.
Therefore,
\begin{eqnarray*}
\|(L_{\rho_\theta} -L_{\rho_{\theta'}})v\|_{\infty}
&\leq& \| v \|_\infty
\left( \sum_{a \in \alpha'} + \sum_{a \in \alpha
\setminus \alpha'}\right)
e^{\varphi(y_a)}
|e^{\pm i\langle\theta,\psi^{\Z}(y_a)\rangle}-e^{\pm i\langle\theta',\psi^{\Z}(y_a)\rangle}|
\\
&\leq& \| v \|_\infty \left(
\sum_{a \in \alpha'}  e^{\varphi(y_a)}
|e^{\pm i\langle\theta,\psi^{\Z}(y_a)\rangle}-e^{\pm i\langle\theta',\psi^{\Z}(y_a)\rangle}|
 + 2\delta\right) \to 2\delta \| v \|_\infty.
\end{eqnarray*}
Since $\delta$ was arbitrary,
$ \|(L_{\rho_\theta} -L_{\rho_{\theta'}})v\|_{\infty}\to 0$.

Now for the $|\cdot |_\beta$ part of the norm,
since $\psi$ is constant on each $a\in\alpha$, $\psi(x)=\psi(y)$
for each $x,y\in\alpha$. Thus,
\begin{align*}
 |(L_{\rho_\theta} -L_{\rho_{\theta'}})v|_{\beta}
&\le |v|_\beta\sum_{ x, y\in a, a\in\alpha} e^{\varphi(x)}
|e^{\pm i\langle\theta,\psi^{\Z}(x)\rangle}-e^{\pm i\langle\theta',\psi^{\Z}(y)\rangle}| \\
&+|v|_\infty\sum_{ x, y\in a, a\in\alpha} |e^{\varphi(x)}-e^{\varphi(y)}|
|e^{\pm i\langle\theta,\psi^{\Z}(x)\rangle}-e^{\pm i\langle\theta',\psi^{\Z}(y)\rangle}|.
\end{align*}
Since by assumption, $|e^{\varphi(x)}-e^{\varphi(y)}|\le Cd_\beta(x,y) e^{\varphi(x)}$ and since $|e^{\pm i\langle\theta,\psi^{\Z}(x)\rangle}-e^{\pm i\langle\theta',\psi^{\Z}(y)\rangle}|\to 0$ as $\theta\to\theta'$, we
can copy the above computation to obtain that $|(L_{\rho_\theta} -L_{\rho_\theta'})|_{\beta}\to 0$ as $\theta\to\theta'$.
\end{proof}

 \subsection{Understanding $L_{\rho_\theta}$ at $\theta=0_d$}
 \label{subsec:thetazero}

Recall equation~\eqref{changebmult1111111111} and the definition of $L_{\rho_{0,\pm 1}}$ in~\eqref{eq:operatorchangebasis1}.
Evaluating at $\theta=0_d$, and taking $n=1$,

\begin{align*}
 (L_{\rho_0})_{euc}\begin{pmatrix}
                    v_1\\
                   v_2 
                  \end{pmatrix}
 = \frac12 \begin{pmatrix}
                   L_{\rho_{0,1}} + L_{\rho_{0,-1}}  &  L_{\rho_{0,1}} - L_{\rho_{0,-1}}  \\
                     L_{\rho_{0,1}} - L_{\rho_{0,-1}} &  L_{\rho_{0,1}} + L_{\rho_{0,-1}}
                  \end{pmatrix}_{euc}\begin{pmatrix}
                    v_1\\
                   v_2 
                  \end{pmatrix}.
\end{align*}
 
 We need to understand the operator $(L_{\rho_0} v)_{euc}$ along  with the individual operators $L_{\rho_{0,\pm 1}}$.
 At $\theta=0_d$, we understand the original Banach space
$\cB_\beta(X,\C^2)$ with the appropriate norm described at the beginning of Section~\ref{section:spectral properties} as a sum of two Banach spaces
\begin{align}\label{eq:isombanachsp}
 \cB_\beta(X,\C^2)\cong\cB_\beta(X,\C)\oplus \cB_\beta(X,\C),
\end{align}
where $\cB_\beta(X,\C)$ is the Banach space of functions $f:X\to\C$,
with  a similar norm. That is, $D_af=\sup_{x,x'\in a,\,x\neq x'}|f(x)-f(x')|/d_\beta(x,x'),\qquad |f|_\beta=\sup_{a\in\{a\}}D_a f$.
The space $\cB_\beta(X,\C)\subset L^\infty(\mu)$ consists of the functions
$f:X\to\C$ such that
$\|f\|_{\cB_\beta(X,\C)}=\|f\|_{\infty(\mu)}+|f|_\beta<\infty$.

To make sense of~\eqref{eq:isombanachsp}, we define a linear
isomorphism 
\begin{align}\label{tildePhi}
\tilde{\Phi}: \cB_\beta(X,\C^2)\to \cB_\beta(X,\C)\oplus\cB_\beta(X,\C),\quad  \tilde{\Phi}v= \begin{pmatrix}
                    v_1\\
                   v_2
                  \end{pmatrix},
\end{align}
for $v\in \cB_\beta(X,\C^2)$, $v_1, v_2\in\cB_\beta(X,\C)$, with inverse $\tilde{\Phi}^{-1}(v_1,v_2)^T=v$.

In other words, we work with the operator
\begin{align}\label{eq:lroeuc}
 (L_{\rho_0})_{euc}=\frac12 \begin{pmatrix}
                   L_{\rho_{0,1}} + L_{\rho_{0,-1}} &  L_{\rho_{0,1}} - L_{\rho_{0,-1}} \\
                     L_{\rho_{0,1}} - L_{\rho_{0,-1}} &  L_{\rho_{0,1}} + L_{\rho_{0,-1}}
                  \end{pmatrix}_{euc}:\cB_\beta(X,\C)\oplus \cB_\beta(X,\C)\to \cB_\beta(X,\C)\oplus \cB_\beta(X,\C) 
\end{align}
and understand the operator $L_{\rho_0}: \cB_\beta(X,\C^2) \to \cB_\beta(X,\C^2)$
as $L_{\rho_0}= \tilde{\Phi}^{-1}(L_{\rho_0})_{euc}\tilde{\Phi}$.\\

\textbf{Understanding of the operators $L_{\rho_{0,\pm 1}}:\cB_\beta(X,\C)\to \cB_\beta(X,\C)$.}
 
 The first part
of~\eqref{eq:assS}, together with the definition of $\varphi$ and the invariance of the measure $\mu$, imply that
\begin{equation}\label{eq:phi}
\varphi\circ S  = \frac{d\mu}{d\mu \circ T} \circ S = \frac{d\mu \circ S}{d\mu \circ T \circ S} = \frac{d\mu}{d\mu \circ T} = \varphi.
\end{equation}
%Indeed, by \eqref{eq:phi} and the fact that
Since $S$ bijectively exchanges $\{ y \in T^{-1}(x) : \eps(y) = 1\}$ with
$\{ y \in T^{-1}(x) : \eps(y) = -1\}$, we have
$$
 L_{\rho_{0, 1}}f(x)= \sum_{y \in T^{-1}(x)} e^{\varphi(y)}f(y)
= \sum_{y \in T^{-1}(x),\, \eps(y) = 1} e^{\varphi(y)}(f(y)+f(Sy)),
$$
and
$$
L_{\rho_{0, -1}}f(x) =\sum_{y \in T^{-1}(x)} e^{\varphi(y)}\eps(y) f(y)
=  \sum_{y \in T^{-1}(x),\, \eps(y) = 1} e^{\varphi(y)}(f(y)-f(Sy)).
$$

\textbf{The eigenvalue $1$ is simple and isolated in the spectrum of $L_{\rho_{0, 1}}$.}
Denote by $1$ the fixed point of the usual transfer operator  $Lf(x)= \sum_{y \in T^{-1}(x)} e^{\varphi(y)} f(y)=L_{\rho_{0, 1}}f(x)$ acting on $\cB_\beta(X,\C)$. By standard arguments (see, for instance, \cite{AD01}), $L1=1$.

Note that $L_{\rho_{0, 1}}1=2\sum_{y \in T^{-1}(x),\, \eps(y) = 1} e^{\varphi(y)}
= \sum_{y \in T^{-1}(x)} e^{\varphi(y)} =L1=1$. 
 So, 
$L_{\rho_{0, 1}}=L$ has the leading value $1$, isolated in the spectrum
of $L_{\rho_0}=L$.\\

\textbf{Spectral radius for $L_{\rho_{0, -1}}$.}
Note that $L_{\rho_{0, -1}}f = L(\eps f)$. 
We argue that the spectral radius $r(L_{\rho_{0, -1}})$ of $L_{\rho_{0, -1}}$ is less than $1$ using assumption~\eqref{apsign}. 

As we argue below, the following assumption on $\eps:X\to\{\pm 1\}$ is automatically satisfied in the set up of Gibbs Markov maps:

 \begin{align}\label{apsign}\nonumber\text{There exists no measurable funtion }  & u:X\to\S^1
 \text{ and } \gamma\in\S^1\text{ with } |\gamma|=1 \text{ so that }\\
  &\eps(y)=\gamma \frac{u(Ty)}{u(y)},\quad \mu\text{-a.e. } y.
 \end{align}

\begin{lemma}\label{lemma:eps}
 Assume~\eqref{apsign}. Then $r(L_{\rho_{0, -1}})<1$.
\end{lemma}

\begin{proof}
 Suppose that $r(L_{\rho_{0, -1}})=1$. Since $L_{\rho_{0, -1}}f = L(\eps f)$ and the operator $L(\eps\cdot)$ is quasi-compact (that is the DF inequality holds in $\cB_\beta(X,\C)$), there exists $f:X\to\C$, $f\ne 0$ and $\gamma\in\S^1$, $|\gamma|=1$ so that
 \[
  L(\eps f)=\gamma f.
 \]
Taking absolute values 
\[
 |f|=|L(\eps f)|\le L(|\eps f|)=L(|f|).
\]
Integrating,
\[
 \int_X |f|\, d\mu\le \int_X L(|f|)\, d\mu=\int_X |f|\, d\mu.
\]
So, $\int_X L(|f|)\, d\mu=\int_X |f|\, d\mu$. Since $\int_X L(|f|)-|f|\, d\mu=0$ and $L(|f|)-|f|\ge 0$, we obtain that 
$L(|f|)=|f|$, $\mu$-a.e. 

We already know that $1$ is a simple  eigenvalue for $L$. The associated eigenspace is generated by the (constant) function $1$. That is, every fixed point of $L$ is a constant, $\mu$-a.e. Since $L(|f|)=|f|$, $\mu$-a.e.\ and $|f|>0$ (because $f\ne 0$), $|f|=c$.

Set $u=\frac{f}{|f|}$, so $|u|=1$, and $f= uc$.
Since $|L(\eps f)|=|f|= L(|f|)$, 
$
c |L(\eps u)|=|f|= c L 1$, so
\[
 |L(\eps u)|=L 1.
\]
Thus,
\begin{align}\label{equalityeps}
 \left|\sum_{y \in T^{-1}(x)} e^{\varphi(y)} \eps(y) u(y)\right|= \sum_{y \in T^{-1}(x)} e^{\varphi(y)} =\sum_{y \in T^{-1}(x)} e^{\varphi(y)}|\eps(y) u(y)|,
\end{align}
where in the last equality we have used that $|\eps(y) u(y)|=1$.
We claim that~\eqref{equalityeps}
implies that
\begin{align}\label{indepy}
 \eps(y) u(y)=\eta(x),\text{ for all } y=T^{-1}x\text { and for }\eta\in\S^1.
\end{align}
To see this claim, suppose by contradiction that there exist $y,y'\in T^{-1}(x)$
so that $Z(y)=:\eps(y) u(y)\ne \eps(y') u(y'):=Z(y')$. 
Since $Z(y), Z(y')\in\S^1$, $\Re(Z(y)\overline{Z(y')})\le 1$ with equality if and only if $Z(y)=Z(y')$.
Squaring equation~\eqref{equalityeps}, expanding the square and taking real parts,
\begin{eqnarray*}
\left|\sum_{y \in T^{-1}(x)} e^{\varphi(y)} \eps(y) u(y)\right|^2
&=& \sum_{y, y' \in T^{-1}(x)}e^{\varphi(y')} e^{\varphi(y)}  Z(y)\overline{Z(y')}\\
&=& \sum_{y \in T^{-1}(x)} e^{2\varphi(y)}+\sum_{\substack{y, y' \in T^{-1}(x)\\y\neq y'}}e^{\varphi(y')} e^{\varphi(y)} \Re( Z(y)\overline{Z(y')}).
\end{eqnarray*}
If $Z(y)\ne Z(y'))$ then $\Re(Z(y)\overline{Z(y')})<1$, violating
the equality in~\eqref{equalityeps}. 

We use~\eqref{indepy} to conclude. Since $L(\eps u)(x)=\gamma u(x) $ (an immediate consequence of $L(\eps f)=\gamma f$)
and $\eps(y) u(y)=\eta(x)$, we have $\eta(x)\sum_{y \in T^{-1}(x)} e^{\varphi(y)}=\eta(x)=\gamma u(x)=\gamma u(Ty)$.
But $\eps(y) u(y)=\gamma u(Ty)$ is not possible due to equality \eqref{apsign}. So, the initial assumption $r(L_{\rho_{0, -1}})=1$ cannot hold.
\end{proof}

As already mentioned, the coboundary condition~\eqref{apsign}
holds in the set of Gibbs Markov maps. This is because of Liv\u{s}ic regularity for Markov maps.
 By Liv\u{s}ic regularity for Markov systems, any measurable $\S^1$
-valued solution admits a H\"older (piecewise H\"older in the Gibbs–Markov setting) representative; hence the cohomological equation can be evaluated on periodic points: see, for instance,~\cite[Theorem 1]{bnt}. More precisely, it follows from~\cite[Theorem 1]{bnt} that $u$ has a piecewise H\"older representative. Now, take a fixed point $p_{T_a}$ of each inverse branch $T_a$. Plugging in~\eqref{apsign}, we get that
$\eps(p_{T_a})=\gamma$.  But this means that $\eps$ is globally constant, which is a contradiction to $\eps=\pm 1$.\\

\textbf{\textbf{The eigenvalue $1$ is simple and isolated in the spectrum of $L_{\rho_{0}}$ as an operator on $\cB_\beta(X,\C^2)$.}}

Recall that $L_{\rho_0}: \cB_\beta(X,\C^2) \to \cB_\beta(X,\C^2)$, where
$L_{\rho_0}= \tilde{\Phi}^{-1}(L_{\rho_0})_{euc}\tilde{\Phi}$
with $(L_{\rho_0})_{euc}$ given in~\eqref{eq:lroeuc}.

A look at the entries of the matrix in~\eqref{eq:lroeuc}, tells us that the spectrum of $L_{\rho_0}$ is given by the union of the spectrum of $L_{\rho_{0, 1}}$ and $L_{\rho_{0,-1}}$.
Recall that $1$ is an isolated in the spectrum of $L_{\rho_{0, 1}}=L$.  By Lemma~\ref{lemma:eps}, $r(L_{\rho_{0, -1}})<1$. 
Thus, $1$ is also a simple, isolated eigenvalue in
the spectrum of of $L_{\rho_{0}}$ as an operator on $\cB_\beta(X,\C^2)$.\\

 \textbf{Spectral projection for $(L_{\rho_{0}})_{euc}$}

 The operator $L_{\rho_{0, 1}}=L$ has $1$ as a leading value, isolated in its spectrum. So, $L_{\rho_{0, 1}}=L$ has a corresponding spectral projection $\Pi_{\rho_{0, 1}} f=(\int_X  f\, d\mu) 1$ (associated with the eigenvalue $1$). In particular, $\Pi_{\rho_{0, 1}} 1= 1$.

Since the spectrum
of $L_{\rho_{0, -1}}$ is strictly inside the unit circle,
the projection $\Pi_{\rho_{0, -1}}$ associated with the part of the spectrum near $1$, is simply $0$, that is $\Pi_{\rho_{0, -1}}\equiv 0$.

  The spectral projection for $(L_{\rho_{0}})_{euc}$ can be written as
 
\begin{align*}
 (\Pi_{\rho_0})_{euc}=\frac12 \begin{pmatrix}
                   \Pi_{\rho_{0,1}} + \Pi_{\rho_{0,-1}} &  \Pi_{\rho_{0,1}} - \Pi_{\rho_{0,-1}} \\
                     \Pi_{\rho_{0,1}} - \Pi_{\rho_{0,-1}} &  \Pi_{\rho_{0,1}} + \Pi_{\rho_{0,-1}}
                  \end{pmatrix}_{euc}:\cB_\beta(X,\C)\oplus \cB_\beta(X,\C)\to \cB_\beta(X,\C)\oplus \cB_\beta(X,\C).
\end{align*}
acting on $v=(v_1, v_2)^T\in \cB_\beta(X,\C)\oplus \cB_\beta(X,\C)$.
Applying $(\Pi_{\rho_0})_{euc}$ to $v$ and recalling that
we get that $\Pi_{\rho_{0,-1}}\equiv 0$,

 \begin{align}\label{projectionzeroeuc}
 (\Pi_{\rho_0})_{euc}\begin{pmatrix}
                    v_1\\
                   v_2
                  \end{pmatrix}=\frac12 \begin{pmatrix}
                   \Pi_{\rho_{0,1}}  &  \Pi_{\rho_{0,1}}  \\
                     \Pi_{\rho_{0,1}} & \Pi_{\rho_{0,1}} 
                  \end{pmatrix}_{euc}\begin{pmatrix}
                    v_1\\
                   v_2
                  \end{pmatrix}
\end{align}

If $v=(1,1)^T$,

\begin{align}\label{projectionzeroeuc1}
 (\Pi_{\rho_0})_{euc}\begin{pmatrix}
                    1\\
                   1
                  \end{pmatrix}=\frac12 \begin{pmatrix}
                   1 &  1 \\
                     1 & 1
                  \end{pmatrix}_{euc}\begin{pmatrix}
                    1\\
                   1
                  \end{pmatrix}.
\end{align}

\subsection{Understanding $L_{\rho_\theta}$ for $\theta$ close to $0_d$}

Via the Banach space isomorphism $\tilde{\Phi}$ defined in~\eqref{tildePhi}, we identify $L_{\rho_\theta}$ with its euclidean matrix representation $(L_{\rho_\theta})_{euc}=\tilde\Phi\circ L_{\rho_\theta}\circ \tilde\Phi^{-1}$ on $\cB_\beta(X,\C)\oplus \cB_\beta(X,\C)$. Since $\tilde\Phi$ is independent of $\theta$, continuity of $L_{\rho_\theta}$ on $\cB(X,\C^2)$  is equivalent to continuity of $(L_{\rho_\theta})_{euc}$.

Lemma~\ref{lemma:Dcont} gives the continuity of $L_{\rho_\theta}$ as an operator on $\cB(X,\C^2)$ for all $\theta\in[0,2\pi)^d\setminus T_d$. The continuity of the operator $(L_{\rho_\theta})_{euc}$ on $\cB_\beta(X,\C)\oplus \cB_\beta(X,\C)$ for all $\theta\in[0,2\pi)^d\setminus T_d$ follows. Moreover, in Euclidean coordinates we can also talk about the continuity of $(L_{\rho_\theta})_{euc}$ at $\theta=0_d$, that is $\|(L_{\rho_\theta})_{euc}-(L_{\rho_0})_{euc}\|_{\cB_\beta(X,\C)\oplus \cB_\beta(X,\C)}\to 0$ as $\theta\to 0_d$. And transporting back with $\tilde{\Phi}$,
we have $\|L_{\rho_\theta}-L_{\rho_0}\|_{\cB_\beta(X,\C^2)}
=\|\tilde\Phi^{-1}((L_{\rho_\theta})_{euc}-(L_{\rho_0})_{euc})\tilde\Phi\|_{\cB_\beta(X,\C^2)}\to 0$. That is, the operator $L_{\rho_\theta}$ on $\cB_\beta(X,\C^2)$ is continuous in $\theta\in[0,2\pi)^d$.

From subsection~\ref{outsidetd}, we know that DF inequality    $\|L_{\rho_\theta}^n v\|_{\cB_\beta(X,\C^2)}\le \gamma_0^n |v|_{\beta}
+ C\|v\|_\infty$ holds for all $\theta\in[0,2\pi)^d\setminus T_d$.
In subsection~\ref{subsec:thetazero}
we have established that $1$ is a simple, isolated eigenvalue in
the spectrum of $L_{\rho_{0}}$ as an operator on $\cB_\beta(X,\C^2)$. We also know that $L_{\rho_\theta}$ as an operator on $\cB_\beta(X,\C^2)$ is continuous at $\theta=0_d$.
Putting all these facts together
and proceeding, as in, for instance, in~\cite[Proof of Theorem 4.1]{AD01},  allows us to write 
\begin{align}\label{eq:spDc0p}
 L_{\rho_\theta}^n v=\lambda_{\rho_\theta}^n\Pi_{\rho_\theta}v+Q_{\rho_\theta}^nv,\text{ for all }\theta\in B_{\delta}(0_d)=\{\theta\in [0,2\pi)^d:|\theta|\le\delta\}=[-\delta,\delta]^d,
\end{align}for some $\delta>0$. 
Here $\lambda_{\rho_\theta}$ is the dominant eigenvalue of $L_{\rho_\theta}$ and $\Pi_{\rho_\theta}Q_{\rho_\theta}=Q_{\rho_\theta}\Pi_{\rho_\theta}=0$ and $\|Q_{\rho_\theta}^n\|_{\cB_\beta(X,\C^2)}\le C\delta_0^n$ for some $\delta_0\in (0,1)$.
Since $L_{\rho_\theta}$ on $\cB_\beta(X,\C^2)$ is continuous in $\theta\in[0,2\pi)^d$, the same applies to $\Pi_{\rho_\theta}$
and $Q_{\rho_\theta}$.

We already know $(L_{\rho_\theta})_{euc}$ is continuous in $\theta\in [0,2\pi)^d$ on $\cB_\beta(X,\C)\oplus \cB_\beta(X,\C)$.
 The same applies to the spectral projection $(\Pi_{\rho_\theta})_{euc}=\tilde\Phi\circ \Pi_{\rho_\theta}\circ \tilde\Phi^{-1}$.
In particular, $\|(\Pi_{\rho_\theta})_{euc}-(\Pi_{\rho_0})_{euc}\|_{\cB_\beta(X,\C)\oplus \cB_\beta(X,\C)}\to 0$ as $\theta\to 0_d$.

Recall~\eqref{eq:spDc0p} and that $(L_{\rho_\theta})_{euc}=\tilde\Phi\circ L_{\rho_\theta}\circ \tilde\Phi^{-1}$ (and the same for the spectral projection $\Pi_{\rho_\theta}$ and $Q_{\rho_\theta}$).  Since $\|(\Pi_{\rho_\theta})_{euc}-(\Pi_{\rho_0})_{euc}\|_{\cB_\beta(X,\C)\oplus \cB_\beta(X,\C)}\to 0$ as $\theta\to 0_d$,
\begin{align}\label{finalspdec}
 (L_{\rho_\theta}^n)_{euc}v=
 \lambda_{\rho_\theta}^n(\Pi_{\rho_0})_{euc}v (1+o(1)) + \tilde\Phi^{-1}(Q_{\rho_\theta}^n v)
\end{align}
 for all $\theta\in B_{\delta}(0_d)=[-\delta,\delta]^d$,
for some $\delta>0$, where $\|\tilde\Phi\circ Q_{\rho_\theta}^n\circ \tilde\Phi^{-1}\|_{\cB_\beta(X,\C)\oplus \cB_\beta(X,\C)}\le C\delta_0^n$,
for some $\delta_0\in (0,1)$.

\subsection{Estimating the eigenvalue $\lambda_{\rho_\theta}$ }
\label{sec:bd}
                      
 We will do the calculations on $\cB(X,\C^2)$, as eigenvalues are not affected by changes of basis.
 Recalling~\eqref{eq:spDc0p} we know that  $\lambda_{\rho_\theta}$ is the leading eigenvalue of $L_{\rho_\theta}:\cB(X,\C^2)\to \cB(X,\C^2)$. 
 Let $\ell_0$ be a linear functional in the dual $\cB(X,\C^2)^*$ 
 so that $L_{\rho_0}^*\ell_0=\ell_0$ and so that $\ell_0(v_{\rho_0})=1$
 where $v_{\rho_0}=(1, 1)^T$ is the eigenvector associated with the eigenvalue
 $\lambda_{\rho_0}=1$. The eigenvector $v_{\rho_0}=(1, 1)^T$
 is  the (unique) invariant density in $\cB(X,\C^2)$.
 Note that $\ell_0$ is the the unique eigenfunctional in the dual space $\cB(X,\C^2)^*$ associated with the eigenvalue $1$.
 The functional $\ell_0$
 is the (unique) invariant measure in $\cB(X,\C^2)^*$.

 The functional $\ell_0$ act on $v$ by $\ell_0(v)=\langle\ell_0, v\rangle=\int_X \langle\xi_0, v\rangle d\mu $, where $\xi_0$ 
 can be determined precisely by
 solving $L_{\rho_0}^*\ell_0=\ell_0$
 and $\ell_0(v_{\rho_0})=1$, using that $\ell_0$ is the unique eigenfunctional satisfying this equation.

\subsubsection{Explicit form of $\xi_0$ as a constant function in $\C^2$}

We claim  $\xi_0=\frac12(1,1)^T$ is the unique function that satisfies 
$\ell_0(L_{\rho_0}v)=\ell_0(v)$.
Let $v=(v_1, v_2)^T$.

Note that
\begin{align*}
 \ell_0(L_{\rho_0}v)&=\int_X \langle\xi_0(x), L_{\rho_0}v(x)\rangle d\mu(x) =\int_X \langle\xi_0(x), \sum_{y\in T^{-1}x} e^{\varphi(y)}\rho_0\circ\psi(y)\, v(y)\rangle d\mu(x)\\
 &=\int_X \sum_{y\in T^{-1}x}e^{\varphi(y)}\langle \xi_0(x),\rho_0\circ\psi(y)\, v(y)\rangle d\mu(x),
\end{align*}
where in the last line we have used the linearity of $\langle \cdot,\cdot\rangle$.

Given $\xi_0=\frac12(1,1)^T$, we verify that $\ell_0(L_{\rho_0}v)=\ell_0(v)$.
Recalling that $\rho_0\circ\psi(y)$ either leaves the two coordinates unchanged or swaps them,
\begin{align*}
 \ell_0(L_{\rho_0}v)=\frac{1}{2}\int_X \sum_{y\in T^{-1}x}e^{\varphi(y)}(v_1(y)+ v_2(y))\,d\mu(x)
 & =\frac{1}{2}\int_X \sum_{y\in T^{-1}x}e^{\varphi(y)}v_1(y)\,d\mu(x)
 +\int_X \sum_{y\in T^{-1}x}e^{\varphi(y)}v_2(y)\, d\mu(x)\\
 &=\frac{1}{2}\int_X (v_1+v_2) d\mu.
\end{align*}
By definition, $\ell_0(v)=\frac{1}{2}\int_X (v_1+v_2) d\mu$.
So, $\ell_0(L_{\rho_0}v)=\ell_0(v)$.
Also, $1=\ell_0(v_{\rho_0})=\frac{1}{2}\int_X (1+1) d\mu$, which is what we required.

The uniqueness of $\xi_0=\frac12(1,1)^T$ follows from the uniqueness of the normalized eigenfunctional $\ell_0$ (in the dual) associated with the eigenvalue $1$.

 \subsubsection{Set up for the calculation of the eigenvalue $\lambda_{\rho_\theta}$}
 
 Start from $\lambda_{\rho_\theta}v_{\rho_\theta}=L_{\rho_\theta}v_{\rho_\theta}$,
where $v_{\rho_\theta}$ is the normalized eigenvector, that is, $
\ell_0(v_{\rho_\theta})=\int_X \langle\xi_0, v_{\rho_\theta}\rangle d\mu =1$.
Applying $\ell_0$ to $\lambda_{\rho_\theta}v_{\rho_\theta}=L_{\rho_\theta}v_{\rho_\theta}$, we get 
\begin{align}\label{eveq}
 \nonumber\lambda_{\rho_\theta}&=\ell_0(L_{\rho_\theta}v_{\rho_\theta})
 =\int_X\langle\xi_0, L_{\rho_\theta} v_{\rho_\theta}\rangle\, d\mu\\
 &=\int_X\langle\xi_0, L_{\rho_\theta} v_{\rho_0}\rangle\, d\mu
 +\int_X\langle\xi_0, (L_{\rho_\theta}-L_{\rho_0}) (v_{\rho_\theta}
 -v_{\rho_0})\rangle\, d\mu.
\end{align}

\textbf{The first term} in~\eqref{eveq} is computed as
\begin{align*}
 \int_X\langle\xi_0, L_{\rho_\theta} (1,1)^T\rangle\, d\mu
 &=\int_{\eps=1}\langle\xi_0, L_{\rho_\theta} (1,1)^T\rangle\, d\mu
 +\int_{\eps=-1}\langle\xi_0, L_{\rho_\theta} (1,1)^T\rangle\, d\mu\\
 &=\int_{\eps=1}\langle\xi_0, (e^{i\langle\theta,\psi^{\Z^d}\rangle}, e^{-i\langle\theta,\psi^{\Z^d}\rangle})^T\rangle\, d\mu
 +\int_{\eps=-1}\langle\xi_0, (e^{-i\langle\theta,\psi^{\Z^d}\rangle}, e^{i\langle\theta,\psi^{\Z^d}\rangle})^T\rangle\, d\mu.
\end{align*}
Recalling $\xi_0=\frac12(1,1)^T$,
\[
 \int_X\langle\xi_0, L_{\rho_\theta} (1,1)\rangle\, d\mu=\int_{\eps=1}\frac12
 (e^{i\langle\theta,\psi^{\Z^d}\rangle}+
 e^{-i\langle\theta,\psi^{\Z^d}\rangle})\, d\mu +\int_{\eps=-1}\frac12
 (e^{i\langle\theta,\psi^{\Z^d}\rangle}+
 e^{-i\langle\theta,\psi^{\Z^d}\rangle})\, d\mu.
\]

Recalling that $\mu(\eps=1)=\mu(\eps=-1)=1/2$, that $\psi^{\Z^d}\in L^{2+\delta^*}(\mu)$ and taking expansions in $\theta$,
we get
\begin{align}\label{firstterm}
 \int_X\langle\xi_0, L_{\rho_\theta} (1,1)^T\rangle\, d\mu= 1-\frac 12 \int_X\langle\theta,\psi^{\Z^d}\rangle^2\, d\mu +O(|\theta|^{2+\delta^*}).
\end{align}

\textbf{For the second  term} in~\eqref{eveq} we need to go to derivatives.

The derivative of $L_{\rho_\theta}$ is well defined in $\cB_\beta(X,\C^2)$.
Recalling the pointwise formula $L_{\rho_\theta}$ (see beginning of
Section~\ref{s:inv}) and that $v= (v_1, v_2)$,
\begin{align*}L_{\rho_\theta} v(x) &=\sum_{y\in T^{-1}x, \eps=1} e^{\varphi(y)}
\begin{pmatrix} 
e^{i\langle\theta,\psi^{\Z^d}(y)\rangle}v_1(y) \\ 
 e^{-i\langle\theta,\psi^{\Z^d}(y)\rangle}v_2(y)
\end{pmatrix}
+\sum_{y\in T^{-1}x,\eps=-1}e^{\varphi(y)}
\begin{pmatrix} 
e^{i\langle\theta,\psi^{\Z^d}(y)\rangle}v_1(y) \\ 
 e^{-i\langle\theta,\psi^{\Z^d}(y)\rangle}v_2(y)
\end{pmatrix}.
\end{align*}

One convenient way to write the derivative in $\theta$ is via the gradient
$\nabla_\theta e^{i\langle\theta,\psi^{\Z^d}\rangle}= i\psi^{\Z^d}e^{i\langle\theta,\psi^{\Z^d}\rangle}$, which can be written coordinate-wise.
Taking a derivative in $\theta$ of $L_{\rho_\theta}$, which we write
as $L_{\rho_\theta}':= \nabla_\theta L_{\rho_\theta} $,
we get 
\begin{align*}(L_{\rho_\theta})' v(x) =
\sum_{y\in T^{-1}x,\eps=1}e^{\varphi(y)}\begin{pmatrix} 
i\psi^{\Z^d}\cdot e^{i\langle\theta,\psi^{\Z^d}(y)\rangle}v_1(y) \\ 
 -i\psi^{\Z^d}\cdot e^{-i\langle\theta\psi^{\Z^d}(y)\rangle}v_2(y)
\end{pmatrix}
+\sum_{y\in T^{-1}x,\eps=-1}e^{\varphi(y)}
\begin{pmatrix} 
i\psi^{\Z^d}\cdot e^{i\langle\theta,\psi^{\Z^d}(y)\rangle}v_1(y) \\ 
-i\psi^{\Z^d}\cdot e^{-i\langle\theta\psi^{\Z^d}(y)\rangle}v_2(y)
\end{pmatrix}.
\end{align*}
Recalling that $\psi^{\Z^d}$ is constant
on partition elements and $\psi^{\Z^d}\in L^{1}(\mu)$
along with the definition of $L^\infty(\mu)$ norm in $\cB_\beta(X,\C^2)$,
\begin{align*}
 \left\|(L_{\rho_\theta})' v\right\|_{\infty}
 \ll\|v\|_\infty\sum_{a\in\alpha}\mu(a)\left|\psi^{\Z^d}|_{a}\right|
\ll
 \|\psi^{\Z^d}\|_{L^1(\mu)}\|v\|_{\infty},
\end{align*}where $\left|\psi^{\Z^d}|_{a}\right|$ is the norm on $\Z^d$ of $\psi^{\Z^d}|_{a}$.
Again using that $\psi^{\Z^d}$ is constant on $a\in\alpha$ and the Gibbs property of $\varphi$ (a simplified version of~\cite[Proposition 12.1]{MelTer17}) and recalling the definition of $|\cdot|_\beta$ seminorm 
in $\cB_\beta(X,\C^2)$ gives
$
 \left|(L_{\rho_\theta})' v\right|_{\beta}
 \ll
 \|\psi^{\Z^d}\|_{L^1(\mu)}\|v\|_{\cB_\beta(X,\C^2)}$. So,
 $
 \left\|(L_{\rho_\theta})' v\right\|_{\cB_\beta(X,\C^2)}
 \ll
 \|\psi^{\Z^d}\|_{L^1(\mu)}\|v\|_{\cB_\beta(X,\C^2)}$.

It follows that the first derivative $(L_{\rho_\theta})'$ evaluated at $0_d$ is well defined in norm. 
 A  calculation using that $\psi^{\Z^d}\in L^{2}(\mu)$ shows that
\begin{align}\label{derivopc2}
 \left\|(L_{\rho_\theta}- L_{\rho_\theta}- (L_{\rho_0})'\theta\right)v \|_{\cB_\beta(X,\C^2)}
\ll |\theta|^{2}\|\psi^{\Z^d}\|_{L^2(\mu)}^2\|v\|_{\cB_\beta(X,\C^2)}\ll|\theta|^2.
\end{align}
 We can also write $(L_{\rho_0})'\theta$
  as 
\begin{align}\label{difwaywritede1}
 [(L_{\rho_0})'\theta] v=L_{\rho_0}\left( (i\langle\theta,\psi^{\Z^d}\rangle \begin{pmatrix}
v_1 \\
- v_2
\end{pmatrix}\right)
\end{align}
and recalling the integral formula for $L_{\rho_\theta}$ (see beginning of
Section~\ref{s:inv}),
\begin{align}\label{difwaywriteder}
\int_X  [(L_{\rho_0})'\theta] v d\mu=
\int_X i\langle\theta,\psi^{\Z^d}\rangle \begin{pmatrix}
v_1 \\
- v_2
\end{pmatrix}\, d\mu.
\end{align}

The same holds for the corresponding eigen-elements, in particular, for the eigenvector $v_{\rho_\theta}=\frac{\Pi_{\rho_\theta} v_{\rho_0}}{\int_X\Pi_{\rho_\theta} v_{\rho_0}d\mu}$.
\begin{align}\label{eq:v1pr0}
 \left\|v_{\rho_\theta}- v_{\rho_0}- (v_{\rho_0})'\theta \right\|_{\cB_\beta(X,\C^2)}\ll |\theta|^{2}.
\end{align}

With the derivatives clarified, we can proceed to the computation of the second term in~\eqref{eveq}. Using~\eqref{derivopc2} and~\eqref{eq:v1pr0}

\begin{align*}
 \int_X\langle\xi_0, & (L_{\rho_\theta}-L_{\rho_0}) (v_{\rho_\theta}
 -v_{\rho_0})\rangle\, d\mu
 =\int_X\langle\xi_0, (L_{\rho_0})'\theta (v_{\rho_\theta}
 -v_{\rho_0})\rangle\, d\mu\\
 & +\int_X\langle\xi_0, (L_{\rho_\theta}-L_{\rho_0}-(L_{\rho_0})'\theta) (v_{\rho_\theta}
 -v_{\rho_0})\rangle\, d\mu \\
 &=\int_X\langle\xi_0, (L_{\rho_0})'\theta\, (v_{\rho_0})')\theta\rangle\, d\mu + \int_X\langle\xi_0, \theta(L_{\rho_0})' (v_{\rho_\theta}
 -v_{\rho_0}-\theta(v_{\rho_0})')\rangle\, d\mu+ O(|\theta|^3)\\
 &=\int_X\langle\xi_0, (L_{\rho_0})'\theta\, (v_{\rho_0})')\theta\rangle\, d\mu +O(|\theta|^3).
\end{align*}

Differentiating in $\theta$ in $v_{\rho_\theta}=\frac{\Pi_{\rho_\theta} v_{\rho_0}}{\int_X\Pi_{\rho_\theta} v_{\rho_0}d\mu}$
and evaluating at $0$, $(v_{\rho_0})'=(\Pi_{\rho_0})' v_{\rho_0}$.

Recall $v_{\rho_0}=\begin{pmatrix}
 1 \\
1
 \end{pmatrix}$.
Using the Cauchy formula and recalling~\eqref{difwaywritede1},
\begin{align*}
 (\Pi_{\rho_0})'v_{\rho_0}&=\frac{1}{2\pi i}
\int_{|u-1|= \delta} (u- L_{\rho_0})^{-1}(L_{\rho_0})'(u- L_{\rho_0})^{-1}v_{\rho_0}du\\
&= \frac{1}{2\pi i}\int_{|u-1|= \delta}(u-1)^{-1} (u- L_{\rho_0})^{-1}L_{\rho_0}\left( (i\langle\theta,\psi^{\Z^d}\rangle \begin{pmatrix}
 1 \\
-1
 \end{pmatrix}\right)du.
\end{align*}

Let $\widetilde \psi^{\Z^d}= \psi^{\Z^d}-\int_X \psi^{\Z^d}\, d\mu$
and note that 
$$L_{\rho_0}\left( (i\langle\theta,\psi^{\Z^d}\rangle \begin{pmatrix}
 1 \\
-1
 \end{pmatrix}\right)=i\langle\theta,\int_X \psi^{\Z^d}\, d\mu\rangle\begin{pmatrix}
 1 \\
-1
 \end{pmatrix} +L_{\rho_0} \left(i\langle\theta,\widetilde \psi^{\Z^d}\rangle\begin{pmatrix}
 1 \\
-1
 \end{pmatrix}\right).$$ Note the first
term $i\langle\theta,\int_X \psi^{\Z^d}\, d\mu\rangle$ is a constant. Cauchy's theorem applied to the constant term/vector gives $0$ inside the formula for $(\Pi_{\rho_0})'v_{\rho_0}$.
Since $\widetilde \psi^{\Z^d}$ has mean zero, $L_{\rho_0} \left(i\langle\theta,\widetilde \psi^{\Z^d}\rangle\begin{pmatrix}
 1 \\
-1
 \end{pmatrix}\right)$ lies in complementary spectral subspace.
 That is,
 \begin{align*}
  \frac{1}{2\pi i}\int_{|u-1|= \delta}(u-1)^{-1} (u- L_{\rho_0})^{-1}L_{\rho_0}\left( (i\langle\theta,\widetilde\psi^{\Z^d}\rangle \begin{pmatrix}
 1 \\
-1
 \end{pmatrix}\right)du
 &=(I- L_{\rho_0})^{-1}(I-\Pi_{\rho_0})L_{\rho_0}\left( (i\langle\theta,\widetilde\psi^{\Z^d}\rangle \begin{pmatrix}
 1 \\
-1
 \end{pmatrix}\right)\\
 &=(I- L_{\rho_0})^{-1}L_{\rho_0}\left( (i\langle\theta,\widetilde\psi^{\Z^d}\rangle \begin{pmatrix}
 1 \\
-1
 \end{pmatrix}\right).
 \end{align*}
Therefore,
\begin{align*}
 (v_{\rho_0})'=(\Pi_{\rho_0})'v_{\rho_0}=\sum_{n\ge 1}L_{\rho_0}^n \left(i\langle\theta,\widetilde \psi^{\Z^d}\rangle \begin{pmatrix}
1 \\
-1
\end{pmatrix}\right).
\end{align*}
This together with~\eqref{difwaywriteder} gives
\begin{align*}
  \int_X\langle\xi_0, (L_{\rho_0})'\theta\, (v_{\rho_0})')\theta\rangle\, d\mu
  &=\sum_{n\ge 1} \int_X\left\langle\xi_0,i\langle\theta,\psi^{\Z^d}\rangle \begin{pmatrix}
 1 \\
-1
 \end{pmatrix}\, L_{\rho_0}^n \left(i\langle\theta,\widetilde \psi^{\Z^d}\rangle \begin{pmatrix}
1 \\
 -1
 \end{pmatrix} \right)\right\rangle\, d\mu\\
 &= \sum_{n\ge 1} \int_X \left\langle\xi_0,
 \begin{pmatrix}
 i\langle\theta,\widetilde\psi^{\Z^d}\rangle \\
 -i\langle\theta,\widetilde\psi^{\Z^d}\rangle
  \end{pmatrix}\,  \begin{pmatrix}
 i\langle\theta,\widetilde \psi^{\Z^d}\rangle \\
 -i\langle\theta,\widetilde \psi^{\Z^d}\rangle
 \end{pmatrix}\circ T^n \right\rangle\, d\mu.
 \end{align*}
Recalling $\xi_0=\frac12\binom{1}{1}$,
\begin{align*}
 \int_X\langle\xi_0, (L_{\rho_0})'\theta\, (v_{\rho_0})')\theta\rangle\, d\mu&
  =-\sum_{n\ge 1} \int_X \langle\theta,\widetilde \psi^{\Z^d}\rangle\cdot \langle\theta,\widetilde \psi^{\Z^d}\rangle\circ T^n d\mu.
\end{align*}

So,
\begin{align*}
 \int_X\langle\xi_0, (L_{\rho_\theta}-L_{\rho_0}) (v_{\rho_\theta}
 -v_{\rho_0})\rangle\, d\mu
 =-\sum_{n\ge 1} \int_X \langle\theta,\widetilde \psi^{\Z^d}\rangle\cdot \langle\theta,\widetilde \psi^{\Z^d}\rangle\circ T^n d\mu + O(|\theta|^3).
\end{align*}

Putting together the first and the second term and recalling~\eqref{eveq},
\begin{align*}
 \lambda_{\rho_\theta}= 1+\frac 12 \int_X\langle\theta,\psi^{\Z^d}\rangle^2\, d\mu-\sum_{n\ge 1} \int_X \langle\theta,\widetilde \psi^{\Z^d}\rangle\cdot \langle\theta,\widetilde \psi^{\Z^d}\rangle\circ T^n d\mu  +O(\theta|^{2+\delta^*}).
\end{align*}
and rewriting in terms of a covariance matrix,
\begin{align}\label{eigenvformula}
 \lambda_{\rho_\theta}= 1-\frac 12 \theta^T \Sigma_1 \theta +o(\theta|^2),
\end{align}
where $\Sigma_1=\int_X \psi^{\Z^d}\, (\psi^{\Z^d})^T\, d\mu+
2 \sum_{n\ge 1} \int_X \widetilde\psi^{\Z^d}\, (\widetilde\psi^{\Z^d})^T\circ T^n\, d\mu$ with $\widetilde \psi^{\Z^d}= \psi^{\Z^d}-\int_X \psi^{\Z^d}\, d\mu$.
 
 The assumption $\psi^{\Z^d}\in L^{2+\delta^*}(\mu)$, so $\widetilde\psi^{\Z^d}\in L^{2+\delta^*}(\mu)$,  is necessary for the sum $\sum_{n\ge 1} \int_X \widetilde\psi^{\Z^d}\, (\widetilde\psi^{\Z^d})^T\circ T^n\, d\mu$ to be convergent.  
 Write $(\widetilde\psi^{\Z^d})_r (\widetilde\psi^{\Z^d}\circ T^j)_s$ for the entries of the matrix.
Since $\widetilde\psi^{\Z^d}$ is locally constant (so piecewise H\"older) and $\psi\in L^{2+\delta^*}$, one shows that for all $r,s$, $\int_X
 (\widetilde\psi^{\Z^d})_r (\widetilde\psi^{\Z^d}\circ T^j)_s\, d\mu$ is decaying exponentially, so the sum $\sum_{j\ge 1}\int 
 (\widetilde\psi^{\Z^d})_r (\widetilde\psi^{\Z^d}\circ T^j)_s\, d\mu$ converges for all $r,s$. This a classical argument: truncate $\widetilde\psi^{\Z^d}$ to make it bounded, apply exponential decay of correlation to the bounded part, control the tail terms using that $\widetilde\psi^{\Z^d}\in L^{2+\delta^*}$ and balance the truncation level with the exponential decay.

\subsection{Final form of $L_{\rho_\theta}^n$ for $\theta\in B_\delta(0_d)=[-\delta,\delta]^d$}
 
Recall~\eqref{finalspdec} and using~\eqref{eigenvformula},
\begin{align}\label{useininv}
 (L_{\rho_\theta}^n)_{euc} v=
   e^{-\frac{n}{2}\theta^T\Sigma_1\theta (1+o(1)} (\Pi_{\rho_0})_{euc}v (1+o(1)) +O(\delta_0^n)
\end{align}
with $(\Pi_{\rho_0})_{euc}v$ given in~\eqref{projectionzeroeuc}.

\section{Reducing to a neighborhood of $0_d$}
\label{sec:red0}

\subsection{Understanding $L_{\rho_\theta}$ for $\theta\in T_d\setminus\{0_d\}$ with $T_d$ as defined in~\eqref{eq:td}}
\label{susebc:pi}
%To understand the problem at $\pi$, we look at  $L_{\rho_\pi}$ as a matrix.

Recall from subsection~\ref{subsec:thetazero} that
 for $\theta\in T_d$, we can write $L_{\rho_\theta} v=L_{\rho_{\theta,1}}v^{+}\oplus L_{\rho_{\theta,-1}}v^{-}$,
 where
 \begin{align}\label{eq:operatorchangebasis}
 \begin{cases}
  L_{\rho_{\theta,1}}v^{+}(x) = \sum_{y \in T^{-1}(x)} e^{\varphi(y)} e^{i\langle \theta,\psi^{\Z^d}(y)\rangle} v^{+}(y), \\[2mm]
  L_{\rho_{\theta,-1}} v^{-}(x)= \sum_{y \in T^{-1}(x)} e^{\varphi(y)} e^{i\langle \theta,\psi^{\Z^d}(y)\rangle} \eps(y)v^{-}(y) .
 \end{cases}
\end{align}
We recall that $v=(v_1, v_2)^T$ and $(v^{+}(x),v^{-}(x))^T=(\frac{v_1(x)+v_2(x)}{\sqrt 2}, \frac{v_1(x)-v_2(x)}{\sqrt 2})$ are the coordinates of $v$ in the diagonal-antidiagonal directions.

\begin{lemma}\label{lemma:noev0}Let $\theta\in T_d\setminus\{0_d\}$.
 The operators $L_{\rho_{\theta,1}}$
 and $L_{\rho_{\theta,-1}}$ have no eigenfunction in $\cB_\beta(X,\C)$ on the unit circle.
\end{lemma}

\begin{proof} Recall that the operators $L_{\rho_{\theta,1}}$
 and $L_{\rho_{\theta,-1}}$ act on functions $v^{+}$ and $v^{-}$, respectively.

 To exclude eigenvalues on the unit circle, it suffices to show that the operator norms of
 $L_{\rho_{\theta,1}}^N$ and
 $L_{\rho_{\theta,-1}}^N$
 are strictly less than $1$ for some $N \in \N$.
 We will compute this operator norm for
 $L_{\rho_{\theta,-1}}^N$; for $L_{\rho_{\theta,1}}^N$
 the computation is the same.

 The norm on $\cB_\beta(X,\C)$ has two components. We start with the symbolic H\"older norm.
 Take $x,x' \in X$ and for $y \in T^{-N}(x)$,
 let $y'$ denote the point in $T^{-N}(x')$ in the same element $a \in \alpha_N$ as $y$.
 Then, because $\eps(y) = \eps(y')$ and $\psi^{\Z^d}(y) = \psi^{\Z^d}(y')$, we get
 \begin{eqnarray*}
  | L_{\rho_{\theta,-1}}^N v^{-}(x) -
   L_{\rho_{\theta,-1}}^N v^{-}(x') |
   &\leq&
   \left| \sum_{y \in T^{-N}(x)} e^{\varphi(y)}  e^{i\langle \theta,\psi^{\Z^d}(y)\rangle}\eps(y) v^-(y)
   - e^{\varphi(y')} e^{i\langle \theta,\psi^{\Z^d}(y')\rangle}\eps(y') v^{-}(y') \right| \\
   &\leq&
   \sum_{y \in T^{-N}(x)} \left| e^{\varphi(y)} v^{-}(y) - e^{\varphi(y')} v^{-}(y') \right| \\
    &\leq&
   \sum_{y \in T^{-N}(x)} \left| e^{\varphi(y)}
   - e^{\varphi(y')} \right| |v^{-}(y)| + e^{\varphi(y')} \left| v^{-}(y) - v^{-}(y') \right| \\
   &\leq& \sum_{y \in T^{-N}(x)} C e^{\varphi(y)}  d_\beta(y,y')
    \| v^{-} \|_{\infty} + e^{\varphi(y')} \beta^N D_\beta(v^{-}) \\
    &\leq& C \beta^N \| v^{-} \|_{\cB_\beta(X,\C)} \leq \frac12 \| v^- \|_{\cB_\beta(X,\C)}
 \end{eqnarray*}
 for $N$ sufficiently large. Here the constant $C$ comes from the H\"older condition on the potential
 $\varphi$. Taking the supremum over all $v^{-} \in \cB_\beta(X,\C)$, we see that $L_{\rho_{\theta,-1}}^N$
  strictly contracts the $| \ |_\beta$-seminorm.

This would already show that $L_{\rho_{\theta,-1}}$ strictly contracts the $\| \ \|_{\cB_\beta(X,\C)}$-norm, except that we need to consider $v^{-}$'s with H\"older seminorm $| v^{-} |_\beta$ very small. In particular, we need to consider $v^{-}$'s such that for each element $a \in \alpha$, either $v^{-}|_a \geq \frac12 \sup |v^{-}|$, or
$v^{-}|_a \leq -\frac12 \sup |v^{-}|$.
For such $v^{-}$, we need to look closer at the $L^\infty$-part of the norm. Here we can exploit the fact that there are always terms with opposite signs, that partially cancel.
Indeed, pick $b, b' \in \alpha$   such that $e^{i \langle\theta, \psi^{\Z^d}|_b \rangle}\eps|_b=1$,  while $e^{i \langle\theta, \psi^{\Z^d}|_{b'} \rangle}\eps|_{b'}=-1$. Here we recall that $\theta\in T_d\setminus\{0_d\}$, so $e^{i
\langle\theta, \psi^{\Z^d}\rangle}=\pm 1$.
Then, writing $y_{b'} \in T^{-N}(x) \cap b'$, we get
\begin{eqnarray*}
  | L_{\rho_{\theta,-1}}^N v^{-}(x) |
   &\leq&
   \left| \sum_{y \in T^{-N}(x)} e^{\varphi(y)}e^{i \langle\theta, \psi^{\Z^d} \rangle}\eps(y)  v^{-}(y)\right| \\
   &\leq&
   \sum_{y_{b'} \neq y \in T^{-N}(x)} e^{\varphi(y)} \sup|v^{-}|
    - 2e^{\varphi(y_{b'})} v^{-}(y_{b'})
    \\ &\leq&  \| v^{-} \|_{\infty} -
    2 e^{\varphi(y_{b'})} \frac12 \sup|v^{-}|
 \leq (1-e^{\varphi(y_{b'})} ) \| v^{-} \|_{\infty}.
 \end{eqnarray*}
Combining the two estimates, we see that $L_{\rho_{\theta,-1}}^N$
  strictly contracts the $\| \ \|_{\cB_\beta(X,\C)}$.
\end{proof}

Using Lemma~\ref{lemma:noev0}, we obtain

\begin{corollary}\label{cor:pi}

There exists $\eps_0\in (0,1)$ so that for all
$\theta\in T_d\setminus\{0_d\}$,
 $\left|\operatorname{Tr}\left(\rho_\theta(\pm 1, r)M_{n,\theta}^*(v,w)\right)\right|\le \eps_0^n$, $v,w\in \cB_\beta(X,\C^2)$.
\end{corollary}

\begin{proof}
Recall that $L_{\rho_\theta}v = L_{\rho_{\theta,1}}v^{+}(x)\oplus  L_{\rho_{\theta,-1}}v^{-}(x)$
as represented in the diagonal-antidiagonal basis.

Procceding as in subsection~\ref{subsec:thetazero} we need to bring this back to the Euclidean basis as to multiply with $\binom{e^{i\langle\theta, r\rangle} \quad 0}{0 \quad e^{-i\langle\theta, r\rangle}}$ and take the trace. 

The next step is to multiply this matrix with
$\binom{e^{i\langle\theta, r\rangle} \quad 0}{0 \quad e^{-i\langle\theta, r\rangle}}$.
But this matrix  is represented w.r.t.\ the Euclidean basis. To multiply the two matrices, we first need to put
$(L_{\rho_\pi}^nv)_{da}$ in Euclidean representation by conjugating via the matrix $U = U^{-1} = \frac{1}{\sqrt{2}} \binom{1 \ 1}{1\ -1}$ (as introduced in~\eqref{eq: change-basis-rep}).

Recalling~\eqref{changebmult1111111111},

\begin{align}\label{changebmult} (L_{\rho_\theta})_{euc}\begin{pmatrix}
                   v_1  \\
                0
                  \end{pmatrix}
 = \frac12 \begin{pmatrix}
                 (L_{\rho_{\theta,1}}^n + L_{\rho_{\theta,-1}}^n)v_1\\(L_{\rho_{\theta,1}}^n - L_{\rho_{\theta,-1}}^n)v_1
                  \end{pmatrix}_{euc},&&(L_{\rho_\theta})_{euc}\begin{pmatrix}
                   0  \\
                v_2
                  \end{pmatrix}
 = \frac12 \begin{pmatrix}
                 (L_{\rho_{\theta,1}}^n - L_{\rho_{\theta,-1}}^n)v_2\\(L_{\rho_{\theta,1}}^n + L_{\rho_{\theta,-1}}^n)v_2
                  \end{pmatrix}_{euc}.
\end{align}

Thus, the trace of the matrix $\rho_\theta(\pm1,r)M_{n,\theta}^*(v,w)$ is 
\begin{align*}
\tr(\rho_\theta(\pm1,r)M_{n,\theta}^*(v,w))=    \cos(\langle\theta, r\rangle)\left[\int [L_{\rho_{\theta,1}}^n + L_{\rho_{\theta,-1}}^n]v_1\cdot w_1\;d\mu +\int [L_{\rho_{\theta,1}}^n + L_{\rho_{\theta,-1}}^n]v_2\cdot w_2\;d\mu\right]
\end{align*}
By Lemma~\ref{lemma:noev0},
the operators $L_{\rho_{\theta,1}}$ and $L_{\rho_{\theta,-1}}$
are contractions, and the conclusion follows.
\end{proof}

\subsection{Understanding $L_{\rho_\theta}$ for $\theta$
outside $B_\delta(0_d)=[-\delta,\delta]^d$. Aperiodicity assumption.}

The following aperiodicity condition is needed to control the behaviour of $L_{\rho_\theta}$ for $\theta$ outside of $B_\delta(0_d)$  and reads as follows:

\begin{itemize}
 \item[(Ap)] For every $\theta\in [0,2\pi)^d\setminus T_d$, the cohomology-like equation
 \begin{align}\label{eq:cohomology-like}
     R \circ T(y) = \begin{cases}
                e^{2i \langle \theta, \psi^{\Z^d}(y) \rangle} R(y) & \text{ if } \eps(y) = 1,\\
                 e^{2i \langle \theta, \psi^{\Z^d}(y) \rangle} R(y)^{-1} & \text{ if } \eps(y) = -1,\\
                \end{cases}
 \end{align}
 has no measurable solution $R:X\to \mathbb{S}^1$.
\end{itemize}
\begin{lemma}\label{lemma:appsi}
Let $\delta>0$ be given in \eqref{finalspdec}.
Let $\delta>0$ be as in \eqref{finalspdec}.
If {\rm (Ap)} holds, then there exists $\delta_1>0$ such that $\|L_{\rho_\theta}^n\|_{\cB_\beta(X,\C^2)}< C\delta_1^n$ for every $\theta\in ([0,2\pi)^d\setminus T_d)\setminus B_\delta(0_d)$.
\end{lemma}

% \begin{lemma}
% There exists $\delta>0$ so that for all $\theta\in \left([0,2\pi)^d\setminus T_d\right)\setminus B_\delta(0_d)$ the following property holds:
% If the following comohomology-like equation
% has no measurable solution $R: X \to \S^1$:
% \begin{equation}\label{eq:cohomology-like}
%  R \circ T(y) = \begin{cases}
%                 e^{2i \langle \theta, \psi^{\Z^d}(y) \rangle} R(y) & \text{ if } \eps(y) = 1,\\
%                  e^{2i \langle \theta, \psi^{\Z^d}(y) \rangle} R(y)^{-1} & \text{ if } \eps(y) = -1,\\
%                 \end{cases}
% \end{equation}
% then there is $\delta_0 \in (0,1)$
% such that $\|L_{\rho_\theta}^n\|_{\cB_\beta(X,\C^2)}< C\delta_0^n$.
% \end{lemma}

\begin{proof}
Assume that there exists $\theta\in \left([0,2\pi)^d\setminus T_d\right)\setminus B_\delta(0_d)$ so that $L_{\rho_\theta} v=\lambda v$ for $|\lambda|=1$ and $v=(v_1,v_2)^T\in\cB_\beta(X,\C^2), v\ne 0$.
Taking the norm of the eigen-value equation, and using that $\rho_\theta$ is a unitary matrix, so $\rho_\theta$ is an isometry,
we get  (for Euclidean norm $|v(x)|$ on $\C^2$):
\begin{eqnarray*}
 |v(x)| = |\lambda v(x)| = \left| (L_{\rho_\theta} v)(x) \right|
 &\leq& \sum_{y \in T^{-1}(x)} e^{\varphi(y)} | \rho_\theta \cdot v(y)| \\
 &=& \sum_{y \in T^{-1}(x)} e^{\varphi(y)} |v(y)| = L(|v|),
\end{eqnarray*}
for the untwisted operator $L$.
Integration gives $\int_X L(|v|) \, d\mu = \int_X |v| \, d\mu$, so
$\int_X L(|v|) - |v| \, d\mu = 0$.
But the integrand is non-negative $\mu$-a.e., so $L(|v|) = |v|$ $\mu$-a.e.
We already know that $L$ satisfies the spectral gap on $(X,\cB_\beta)$ and that $1$ is a fixed point for  $L$ with associated eigenvalue $1$. The eigenspace associated with the eigenvalue $1$ is generated by the constant function $1$. That is, every fixed point of $L$ is a constant, $\mu$-a.e. Given $|v|$ is such a fixed point,  $|v|$ is constant, $\mu$-a.e. From now on, we scale $v$ such that $|v| = 1$ $\mu$-a.e.

Now that we know that $v(y) \in \S^1$, the eigenfunction equation
$v(x) = \sum_{y \in T^{-1}(x)} e^{\varphi(y)} \rho_\theta(\psi(y)) v(y)$
can only be satisfied if $\lambda v(x) = \rho_\theta(\psi(y)) v(y)$ for all $y \in T^{-1}(x)$ and $\mu$-a.e.\ $x$. Indeed, because $\sum_{y \in T^{-1}(x)} e^{\varphi(y)} = 1$,
only when all summands $\rho_\theta(\psi(y)) v(y)$ have the same direction,
the sum can equal $\lambda v(x)$.

Using the explicit form of $\rho_\theta$, and separating the two components of $v = (v_1,v_2)$, we get
\begin{equation}\label{eq:4equations}
\begin{cases}
     \begin{cases}
        \lambda v_1 \circ T(y) = e^{i \langle \theta, \psi^{\Z^d}(y) \rangle} v_1(y)\\
        \lambda v_2 \circ T(y) = e^{-i \langle \theta, \psi^{\Z^d}(y) \rangle} v_2(y)
      \end{cases}
        & \text{ if } \eps(y) = 1,\\[8mm]
      \begin{cases}
        \lambda v_2 \circ T(y) = e^{-i \langle \theta, \psi^{\Z^d}(y) \rangle} v_1(y)\\
        \lambda v_1 \circ T(y) = e^{i \langle \theta, \psi^{\Z^d}(y) \rangle} v_2(y)
      \end{cases}   & \text{ if } \eps(y) = -1.\\
                \end{cases}
\end{equation}
If the set $A = \{ y \in X : v_2(y) = 0\}$ has positive measure,
it follows that $\bigcup_{n \geq 0} T^n(A)$ has full measure by exactness of $T$. Exactness is a consequence of mixing.
From \eqref{eq:4equations} it then follows that $v_1(y) = v_2(y) = 0$
for every $y \in \bigcup_{n \geq 0} T^n(A)$, $\mu$-a.e.
This contradicts that $|v(y)| = 1$ $\mu$-a.e.
Hence, we can define $R(y) = v_1(y)/v_2(y)$ for $\mu$-a.e.\ $y \in X$,
and dividing the two pairs of equations in \eqref{eq:4equations},
we see that $R$ must satisfy the equation \eqref{eq:cohomology-like}, against our assumption.
This proves the lemma.
\end{proof}

\subsection{Final expression for $L_{\rho_\theta}^n$ for 
$\theta\in [0,2\pi)^d\setminus T_d$}

Assume that {\rm (Ap)}, that is~\eqref{eq:cohomology-like} holds.
Then,
by~\eqref{useininv} together with Corollary~\ref{cor:pi}, we have 
\begin{align}\label{finalfinal}
 (L_{\rho_\theta}^n)_{euc} v=
   e^{-\frac{n}{2}\theta^T\Sigma_1\theta (1+o(1))} (\Pi_{\rho_0})_{euc}v (1+o(1)) +O(\delta_0^n)
\end{align}
for every $\theta\in B_\delta(0_d)$, with $(\Pi_{\rho_0})_{euc}v$ given in~\eqref{projectionzeroeuc}, and $\|L_{\rho_\theta}^n\|<C\delta_1^n$ for every $\theta\in ([0,2\pi)^d\setminus T_d)\setminus B_\delta(0_d)$, where $\delta_1>0$ is given in Lemma~\ref{lemma:appsi}.

\section{Completing the proof of the main results}
\label{sec:compdih}
\subsection{Completing the Proof of Theorem~\ref{prop:LCLTGM}}

Recall from~\eqref{eq:invdihGM} that
\begin{align*}
 \mu(\{x\in X: \psi_n(x)=(\pm 1,r)\})=\frac{1}{(2\pi)^d}\int_{[0,2\pi)^d\setminus T_d}
 \operatorname{Tr}\left(\rho_\theta(\pm 1,r)\, M_{n,\theta}^*(1, 1)^T\, d\mu\right)\, d\theta,
\end{align*}
where
\begin{align*}
      M_{n,\theta}^*(1,1)^T=\left[\int_X L_{\rho_\theta}^n (1e_1)\, d\mu\quad\quad  \int_X L_{\rho_\theta}^n (1e_2)\, d\mu\right]
     \end{align*}
Since the column vectors are in Euclidean basis, we use~\eqref{finalfinal}.
So we can write 

\begin{align*}
      M_{n,\theta}^*(1,1)^T=\left[ e^{-\frac{n}{2}\theta^T\Sigma_1\theta (1+o(1)} (\Pi_{\rho_0})_{euc}(1e_1)+O(\delta_0^n) \, d\mu\quad\quad  e^{-\frac{n}{2}\theta^T\Sigma_1\theta (1+o(1)} (\Pi_{\rho_0})_{euc}(1e_2)+O(\delta_0^n)\, d\mu\right]
     \end{align*}
     
But recalling~\eqref{projectionzeroeuc1}, 
$
 (\Pi_{\rho_0})_{euc}(1 e_1)
                  =\frac12\begin{pmatrix}
                   1  \\
                  1
                  \end{pmatrix}
                  $
 and $
 (\Pi_{\rho_0})_{euc}(1e_2)=\frac12\begin{pmatrix}
                   1  \\
                  1
                  \end{pmatrix}
                  $.

Thus,
\begin{align*}
 \mu(&\{x\in X: \psi_n(x)=(\pm 1,r)\})\\
 &=\frac{1}{(2\pi)^d}\int_{[0,2\pi)^d\setminus T_d}
 e^{-\frac{n}{2}\theta^T\Sigma_1\theta (1+o(1)}\operatorname{Tr}\left(\rho_\theta(\pm 1,r)\frac12
 \begin{pmatrix}
 1 &  1 \\
  1 & 1
\end{pmatrix}\right)\, d\theta+O(\delta^n).
\end{align*}
So, $\mu(\{x\in X: \psi_n(x)=(\pm 1,r)\})=\frac{1}{(2\pi)^d}\int_{[0,2\pi)^d\setminus T_d}e^{-\frac{n}{2}\theta^T\Sigma_1\theta (1+o(1))}\cos\left(\langle\theta, r\rangle\right) d\theta+O(\delta^n)$.
By the change of variables $\theta\to\sigma/\sqrt{n}$,
$\left|n^{d/2}\mu(\{x\in X: \psi_n(x)=(\pm 1,r)\})- \Phi_1\left(\pm \frac{r}{\sqrt{n}}\right)\right|\to 0$, where
$\Phi_1$ is the density of Gaussian with mean $0$ and variance $\Sigma_1^2$.

The last statement in Proposition~\ref{prop:LCLTGM} (a.k.a. a mixing LCLT)
follows the same way starting from~\eqref{eq:invdihGMfun},
but working with 
\begin{align*}
      M_{n,\theta}^*(v, w)=M_{n,\theta}^*(v_1, w_1)+ M_{n,\theta}^*(v_2, w_2)
     \end{align*}
     with $M_{n,\theta}^*(v_1, w_1)$ and $M_{n,\theta}^*(v_2, w_2)$
     given in~\eqref{eq:mnstar},
instead of $M_{n,\theta}^*(1,1)^T$. In particular, in this case we use~\eqref{projectionzeroeuc} instead of~\eqref{projectionzeroeuc1}.

\begin{remark}\label{rmk:cor2,1}
For use in the proof of Theorem~\ref{cor:2} below we record the following.

Let $e=(1,0_{\Z^d})$ be the neutral element of $G_d$ and note that
$\rho_\theta(e)=I$. Starting from~\eqref{eq:invdihGM2},
 
 \begin{align*}
 1_{\{x\in X: \psi_n(x)=e\}}=\frac{1}{(2\pi)^d}\int_{\theta\in [0,2\pi)^d\setminus T_d}
 \operatorname{Tr}\left(
 M_{n,\theta}(1,1)^T(x)\right)\,d\theta,
\end{align*}
where
\begin{align*}
      M_{n,\theta} (1,1)^T(x)=\left[L_{\rho_\theta}^n (1e_1)(x)\quad\quad  L_{\rho_\theta}^n (1e_2)(x)\right].
     \end{align*}

 Proceeding as in the previous proof we get
 \begin{align}\label{eq:usecor11111}
  1_{\{x\in X: \psi_n(x)=e\}}&=c_d n^{-d/2}(1+o(1))+ Q^n(1,1)^T,
 \end{align}
 where $Q^n(1,1)^T= Q_{\rho_\theta}^n (1e_1)+Q_{\rho_\theta}^n (1e_2) $  and $c_d$ is a Gaussian density with mean $0$ and variance $\Sigma_1^2$ evaluated at $0$. In particular, $\|Q^n\|_{\cB(X,\C^2)}= O(\delta^n)$.
\end{remark}

\subsection{Proof of Theorem~\ref{cor:2}}
\label{subsec:prcor2}

A renewal equation was exploited in~\cite{Tho16} for studying
the big tail of the first return time, called $\varphi$  there, to the origin of a $\Z^d$ extension
of a Gibbs Markov semiflow (a suspension semiflow over a mixing GM map with $L^1$ roof function). 
For a refinement of this procedure, in order to obtain the small tail of $\varphi$ without second moment assumption, which further leads to (Krickeberg) mixing (though resuming to $\Z$ extensions instead of $\Z^2$) we refer to~\cite{T22}.
More precisely, given that  $\mu_0$ is the induced  measure
(on $X\times\{0\}\times\{0\}$),~\cite{Tho16} provides the asymptotic of 
$\mu(\varphi\ge t)$, as $t\to\infty$, while~\cite{T22} gives the asymptotic
of $\mu(t\le \varphi\le t+1)$, as $t\to\infty$. 

For the proof of Theorem~\ref{cor:2}, we need  a discrete version of~\cite[Lemma 1.8]{Tho16}, the so called renewal equation for $\Z^d$ extensions of Gibbs Markov flow. This renewal equation can be traced back to~\cite{PoSh}. Since the roof function (called $r$ there ) plays no role for $G_d$ extensions of GM maps we provide a proof, though the gist of the proof is the same. The same applies to the arguments displayed after Lemma~\ref{lemm:ren}; they are a discrete version of the arguments used in~\cite[Proof of Proposition 1.2]{Tho16}.

We recall the notation in the statement of Theorem~\ref{cor:2}.
Recall that $(X,T,\alpha,\mu)$ is a mixing GM map, $Y=X\times\{e\}$,
$\tau$ the first return time of $T_\psi$ to $Y$, that is $\tau(x)=\min\{n\ge 1:\psi_n(x)=e\}$,
and that $\tilde\nu=\mu\otimes\delta_e$. Note that $\{\tau=0\}=\emptyset$, and, by convention, we write 
$1_{\{\tau=0\}}=0$. Also, we say that $\psi_0\equiv e$.

Let $T_\psi^Y=T_\psi^\tau$ be the first return map to $Y$. Recall that
$L_{\rho_0}$ is the transfer operator of $(X,T,\alpha,\mu)$ acting on the Banach space $\cB(X,\C^2)$. The operator renewal identity we are after reads as

\begin{lemma}\label{lemm:ren}
 Let $\tilde L$ be the transfer operator the first return map $T_\psi^Y$. Let $v:X\to\C^2$, $v\in L^1(\mu)$.
 Define $\tilde L_z v= \tilde L(z^\tau v)=\sum_{n\ge 0} z^n L^n1_{\{\tau=n\}}$.
 Then for all $z\in\bar\D$,
 \begin{align}\label{eq:reneq}
  \sum_{m\ge 0}\tilde L_z^m v=\sum_{n\ge 0} z^n L^n1_{\{\psi_n=e\}} v.
 \end{align}
\end{lemma}

\begin{proof}
 We first look at the RHS of~\eqref{eq:reneq}. Decompose $\{\psi_n=e\}$
 according to the number of first-return blocks. That is, define
 \[
  \tau_1:=\tau,\quad \tau_{j+1}(x)= \tau_{j}(x)+\tau(T^{\tau_j(x)}(x))
 \]
 and note that $\tau_m$ is the $m$-th return of the cocycle $\psi_n$ to $e$.
 For each fixed $n\ge 1$, write
 \[
  \{\psi_n=e\}=\cup_{m\ge 1}\{\tau_m=n\}.
 \]
This is because the 
$m$-th return to $e$ happens exactly at time 
$n$ for a unique $m$. In other words, the sequence of the successive return times is strictly increasing ($\tau_{j+1}(x)> \tau_{j}(x)$), so a given time $n$
 can be the 
$m$-th return for a single $m$.

Thus, for any $v,w:X\to\C^2$, $v\in L^1(\mu)$ and $w\in L^\infty$,
\begin{align}\label{eq:it1}
 \sum_{n\ge 0} z^n \int_X L^n1_{\{\psi_n=e\}} v w\, d\mu 
 =\sum_{n\ge 0} z^n \sum_{m\ge 1}\int_X L^n1_{\{\tau_m=n\}} vw\circ T^m\, d\mu.
\end{align}

On the other hand, by definition,
$
\int_X \tilde{L}_z v\, w\, d\mu= \sum_{n\ge 0} z^n \int_X L^n1_{\{\tau=n\}} v w\, d\mu$.
Applying this equation iteratively (and using induction on $m$), we get 

\begin{align}\label{eq:it2}
\int_X \tilde{L}_z^m v\, w\, d\mu=\int_X \tilde{L}(z^{\tau_m} v)\, w\, d\mu =\sum_{n\ge 0} z^n \int_X L^n1_{\{\tau_m=n\}} v w\, d\mu. 
\end{align}
Summing over $m\ge 1$ in~\eqref{eq:it2} and adding the term with $m=0$, we get 
\begin{align}\label{eq:it3}
\sum_{m\ge 0}\int_X \tilde{L}_z^m v\, w\, d\mu=\int_X v w\, d\mu+\sum_{m\ge 1}\int_X \tilde{L}_z^m v\, w\, d\mu =\sum_{n\ge 0} z^n \sum_{m\ge 0}\int_X L^n1_{\{\tau_m=n\}} v w\, d\mu.
\end{align}
For $z\in\D$, the conclusion follows from~\eqref{eq:it3} and~\eqref{eq:it1}. 
\end{proof}

Take $v_{\rho_0}=(1,1)^T$ (the invariant density written on $\cB(X,\C^2)$) in~\eqref{eq:reneq}. Recall from subsection~\ref{sec:bd} that 
$\ell_0(v_{\rho_0})=\int_X \langle\xi_0, v\rangle d\mu=1$
, where $\xi_0=\frac{1}{2} (1,1)^T$.

Recall $1_{\{\psi_0=e\}}=1$.
 Recalling equation~\eqref{eq:usecor11111} in Remark~\ref{rmk:cor2,1}, we write
\begin{align}\label{eq:useren3}
\sum_{m\ge 0}\tilde L_z^m(1,1)^T &= (1,1)^T + c_d \sum_{n\ge 1} z^n n^{-d/2}(1+o(1))+
\sum_{n\ge 1} z^n Q^n (1,1)^T\\
 \nonumber &=:G(z)(1,1)^T (1+o(1)),
 z\in\bar\D\setminus\{1\},
\end{align}
where $G(z)= 1 + c_d \sum_{n\ge 1} z^n n^{-d/2}(1+o(1))$. 
We recall from Remark~\ref{rmk:cor2,1} that 
$\|Q^n\|_{\cB(X,\C^2)}= O(\delta^n)$, which justifies the definition of $G(z)$.

Multiplying with $I-\tilde L_z$ on both sides of~\eqref{eq:useren3},
we obtain
\begin{align*}
 (I-\tilde L_z)\sum_{m\ge 0}\tilde L_z^m (1,1)^T=G(z) (1+o(1)) (I-\tilde L_z) (1,1)^T.
\end{align*}
Since $\sum_{m\ge 0}\tilde L_z^m (1,1)^T=(I-\tilde L_z)^{-1}(1,1)^T$ for $z\ne 1$ (due to the fact $G(z)$ is well defined for all $z\ne 1$),
\(
 (1,1)^T=G(z) (1+o(1))(I-\tilde L_z) (1,1)^T.
\)
Integrating over $Y=X\times\{e\}$ w.r.t. $\tilde\nu=\ell_0\otimes (v_{\rho_0}\delta_e)$,

\begin{align*}
 1 &=G(z) \int_Y (I-\tilde L_z) (1,1)^T\, d\tilde\nu(1+o(1))
 =G(z)(1+o(1))\sum_{n\ge 1} (1-z^n) \tilde\nu(\tau=n).
\end{align*}
But,
\begin{align*}
 \sum_{n\ge 1} (1-z^n) \tilde\nu(\tau=n)&=\sum_{n\ge 1} (1-z^n) \tilde\nu(\tau\ge n)-\sum_{n\ge 1} (1-z^n) \tilde\nu(\tau\ge n-1)\\
 &=\sum_{n\ge 1} (1-z^n) \tilde\nu(\tau\ge n)-\sum_{n\ge 2} (1-z^{n-1}) \tilde\nu(\tau\ge n)=(1-z)\sum_{n\ge 1} z^{n-1} \tilde\nu(\tau\ge n)\\
 &=\frac{1-z}{z}\sum_{n\ge 1} z^{n} \tilde\nu(\tau\ge n)=(1-z)\sum_{n\ge 0}z^n\tilde\nu(\tau\ge n+1).
\end{align*}

Thus,
\begin{align}\label{eq:f}
 \frac{1+o(1)}{(1-z)G(z)}=\frac{1}{z}\sum_{n\ge 1} z^{n} \tilde\nu(\tau\ge n)=\sum_{n\ge 0} z^n\tilde\nu(\tau\ge n+1).
\end{align}
Finally note that $\tilde\nu(\tau=n)=\ell_0\otimes(v_{\rho_0}\delta_e)(\tau=n)=\mu\otimes\delta_e(\tau=n)$.

With~\eqref{eq:f} we can complete\\

\begin{proof}{~of Theorem~\ref{cor:2}}.
Recall $G(z)= 1 + c_d \sum_{n\ge 1} z^n n^{-d/2}(1+o(1))$.

\begin{itemize}
 \item If $d=1$, $G(z)=c_1(1-z)^{-1/2}(1+o(1))$, for some $c_1>0$, as $z\to 1$.
 So, $\frac{1+o(1)}{(1-z)G(z)}=c_1^{-1}(1-z)^{-1/2}(1+o(1))$.
 A computation of the $n$-th Taylor coefficient of $f(z)=\sum_n f_n z^n=c_1^{-1}(1-z)^{-1/2}$ on the circle $\{e^{-u}e^{it}:t\in[-\pi,\pi)\}$
 with $e^{-u}=e^{-1/n}, n\ge 1$ (see, for instance,~\cite[Proof of Theorem 1.1]{Terhesiu16} with $\beta=1/2$ there) gives that $f_n=C_1 n^{-1/2}(1+o(1))$.
 Using the second equality in~\eqref{eq:f}, 
 $\tilde\nu(\tau\ge n+1)\sim f_n=C_1 n^{-1/2}(1+o(1))$.
 
  \item If $d=2$, $G(z)=1 + c_2 \sum_{n\ge 1} z^n n^{-1}(1+o(1))
  =-c_2\log\left(1-z\right)(1+o(1))$ as $z\to 1$.
  Using the first equality in~\eqref{eq:f} , $\sum_{n\ge 1} z^n\tilde\mu(\tau\ge n)=\frac{z(1+o(1))}{(1-z)G(z)}=-c_2^{-1}\frac{z}{(1-z)\log\left(1-z\right)}(1+o(1))$. The presence of $z$ makes the pole at $0$ disappear.
  An argument based on the transfer theorem in~\cite{FO} shows that
  the $n$-th Taylor coefficient of $\frac{z}{\log\left(1-z\right)}$ is 
  $\frac{1}{n(\log n)^2}(1+o(1))$. Since all the Taylor coefficients of $\frac{1}{1-z}$ are $1$, we obtain by the convolution of the Taylor coefficients $\frac{1}{n(\log n)^2}(1+o(1))$ and $1$ that
  $\tilde\mu(\tau\ge n)=\frac{C_2 (1+o(1))}{\log n}$.\\

 \item If $d\ge 3$, $G(1)=a<\infty$. As a consequence, $\lim_{n\to\infty}\tilde\nu(\tau\ge n)$ is a constant and $\sum_n \tilde\nu(\tau=n)<1$.

\end{itemize}

\end{proof}

\end{document}